\documentclass[10pt,leqno]{amsart}
\usepackage{amsmath}
\usepackage{amsfonts}
\usepackage{amssymb}
\usepackage{graphicx}
\usepackage{color}
\usepackage{hyperref}
\usepackage{xcolor}
\theoremstyle{plain}
\newtheorem{theorem}{Theorem}[section]
\newtheorem{proposition}[theorem]{Proposition}
\newtheorem{corollary}[theorem]{Corollary}
\newtheorem{lemma}{Lemma}[section]

\theoremstyle{definition} 

\newtheorem*{definition*}{Definition}
\newtheorem*{remark*}{Remark}
\newtheorem{remark}{Remark}[section]

\newcommand{\R}{\mathbb{R}}

\def \b {\beta}

\def\Ric{\text{Ric}}

\def\a{\alpha}
\def\l{\lambda}

\def\ve{\varepsilon}
\def\p{\partial}

\def\R{\mathbb{R}}

\def\vp{\varphi}

\def\k{\kappa}
\def\Ric{\operatorname{Ric}}
\def\Rm{\operatorname{Rm}}

\def\tr{\operatorname{tr}}

\def\diam{\operatorname{diam}}
\def\Diam{\operatorname{Diam}}

\def\n{\nabla}

\numberwithin{equation}{section}
\makeatletter
\newcommand*\owedge{\mathpalette\@owedge\relax}
\newcommand*\@owedge[1]{%
  \mathbin{%
    \ooalign{%
      $#1\m@th\bigcirc$\cr
      \hidewidth$#1\m@th\wedge$\hidewidth\cr
    }%
  }%
}
\makeatother

\begin{document}

\title[Li-Yau-Hamilton Estimates and Parabolic Frequency]{Matrix Li-Yau-Hamilton Estimates under Ricci Flow and Parabolic Frequency}

\author[Li]{Xiaolong Li}\thanks{The first author's research is partially supported by Simons Collaboration Grant \#962228 and a start-up grant at Wichita State University}
\address{Department of Mathematics, Statistics and Physics, Wichita State University, Wichita, KS, 67260, USA}
\email{xiaolong.li@wichita.edu}

\author[Zhang]{Qi S. Zhang}\thanks{The second author's research is partially supported by Simons Collaboration Grant \#710364.}
\address{Department of Mathematics, University of California, Riverside, Riverside, CA, 92521, USA}
\email{qizhang@math.ucr.edu}



\begin{abstract}
We prove matrix Li-Yau-Hamilton estimates for positive solutions to the heat equation and the backward conjugate heat equation, both coupled with the Ricci flow. 
We then apply these estimates to establish the monotonicity of parabolic frequencies up to correction factors. As applications, we obtain some unique continuation results under the nonnegativity of sectional or complex sectional curvature. 
\end{abstract}

\maketitle

\section{Introduction}

\subsection{Matrix Li-Yau-Hamilton Estimates}
In their seminal paper \cite{LY86}, P. Li and S.-T. Yau developed fundamental gradient estimates for positive solutions to the heat equation on a Riemannian manifold. In particular, they proved that if $u:M^n\times [0,\infty) \to \R$ is a positive solution to the heat equation 
\begin{equation}\label{heat equation fixed metric}
u_t-\Delta_g u=0,
\end{equation}
on an $n$-dimensional complete Riemannian manifold $(M^n,g)$ with nonnegative Ricci curvature, then 
\begin{equation}\label{eq Li-Yau gradient estimate}
    \frac{u_t}{u}-\frac{|\n u|^2}{u^2} +\frac{n}{2t} = \Delta \log u +\frac{n}{2t} \geq 0
\end{equation}
for all $(x,t)\in M\times (0,\infty)$. 
Remarkably, the equality in \eqref{eq Li-Yau gradient estimate} is achieved on the heat kernel on the Euclidean space $\R^n$. The Li-Yau estimate is also called a differential Harnack inequality since integrating it yields a sharp version of the classical Harnack inequality originated from Moser \cite{Moser64}. The Li-Yau estimate and its generalizations in various settings provide a versatile tool for studying the analytical, topological, and geometrical properties of manifolds (see for instance the classical books \cite{SYbook} and \cite{Libook}). 

Under the stronger assumption that $(M^n,g)$ has nonnegative sectional curvature and parallel Ricci curvature, Hamilton \cite{Hamilton93} extended the Li-Yau estimate to the full matrix version
\begin{equation}\label{eq Hamilton matrix estimate}
    \n_i \n_j \log u +\frac{1}{2t} g_{ij} \geq 0.
\end{equation}
Note that the trace of \eqref{eq Hamilton matrix estimate} is \eqref{eq Li-Yau gradient estimate}. Later on, Chow and Hamilton \cite{CH97} further extended \eqref{eq Li-Yau gradient estimate} and \eqref{eq Hamilton matrix estimate} to the constrained case under the same curvature assumptions, and discovered new linear Harnack estimates. 

When $(M^n,g)$ is a complete K\"ahler manifold with nonnegative bisectional curvature, Cao and Ni \cite{CN05} proved the matrix inequality 
\begin{equation}\label{eq Cao-Ni matrix estimate}
    \n_\a \n_{\bar{\b}} \log u +\frac{1}{t}g_{\a\bar{\b}} \geq 0,
\end{equation}
and they called it a matrix Li-Yau-Hamilton estimate. Inspired by Chow's interpolation consideration \cite{Chow98} of Li-Yau's and Hamilton's Harnack inequalities on a surface,  Chow and Ni \cite[Theorem 2.2]{Ni07} proved that if $(M^n, g(t))$ is a complete solution to K\"ahler-Ricci flow with bounded nonnegative bisectional curvature and if $u$ is a positive solution to the forward conjugate heat equation
\begin{equation}\label{forward conjugate heat equation}
u_t-\Delta_{g(t)} u=Ru,
\end{equation}
then 
\begin{equation}\label{eq Ni matrix estimate forward}
R_{\a \bar{\b}} + \n_\a \n_{\bar{\b}} \log u +\frac{1}{t}g_{\a\bar{\b}} \geq 0.
\end{equation}
The equality holds if and only if $(M^n,g(t))$ is an expanding K\"ahler-Ricci soliton.
These results were generalized to the constrained case by Ren, Yao, Shen, and Zhang \cite{RYSZ15}. 

When the metrics are evolved by the Ricci flow (see \cite{Hamilton82} or \cite{CLN})
\begin{equation*}
    \p_t g =-2\Ric,
\end{equation*}
Perelman \cite{Perelman1} discovered a spectacular differential Harnack estimate for the fundamental solution to the backward conjugate heat equation
\begin{equation}\label{backward conjugate heat equation}
u_t+\Delta_{g(t)} u =R u,
\end{equation}
where $R$ denotes the scalar curvature of $(M^n,g(t))$. Astoundingly, his estimate did not require any curvature conditions. 
For more information on matrix Li-Yau-Hamilton estimates and differential Harnack estimates, as well as their important applications in geometry, we refer the reader to Chow's survey \cite{Chow22}, Ni's survey \cite{Ni08Survey}, and the monographs \cite[Chapters 15-16]{Chowbookpart2} and \cite[Chapters 23-26]{Chowbookpart3} by Chow, etc.

In this paper, we first extend Hamilton's matrix estimate \eqref{eq Hamilton matrix estimate} for static metrics to the Ricci flow case. Our estimate does not require the parallel Ricci curvature condition and thus should be more applicable. 
\begin{theorem}\label{thm matrix Harnack heat equation}
Let $(M^n,g(t))$, $t\in [0,T]$, be a complete solution to the Ricci flow. 
Let $u:M^n \times [0,T] \to \R$ be a positive solution to the heat equation 
\begin{equation}\label{heat equation} 
u_t -\Delta_{g(t)}u=0.
\end{equation}
Suppose that $(M^n,g(t))$ has nonnegative sectional curvature and $\Ric\leq \kappa g$ for some constant $\kappa>0$. Then
\begin{equation}\label{matrix estimate K geq 0}
        \n_i \n_j 
        \log u + \frac{\kappa }{1-e^{-2\kappa t}} g_{ij} \geq 0,
\end{equation}
    for all $(x,t)\in M\times (0,T)$. 
\end{theorem}

\begin{remark}
Note that \eqref{matrix estimate K geq 0} can be restated equivalently as 
\begin{equation*}
    \n_i \n_j u +\frac{\kappa u}{1-e^{-2\kappa t}} g_{ij} +\n_i u V_j +\n_j u V_i +u V_iV_j \geq 0
\end{equation*}
for any vector field $V$ by choosing the optimal vector field $V=-\n \log u$. Other matrix Li-Yau-Hamilton estimates also admit such equivalent restatements. 
\end{remark}

\begin{remark}
Note that \eqref{matrix estimate K geq 0} is asymptotically sharp as $t\to 0^+$. To see this, one notices
\begin{equation}\label{eq compare factor}
\frac{1}{2t} < \frac{\kappa }{1-e^{-2\kappa t}} < \frac{1}{2t}+\k
\end{equation}
for all $t>0$ and $\k>0$. Applying Theorem \ref{thm matrix Harnack heat equation} to $(M^n,g)=(\R^n,\delta_{ij})$ and letting $\kappa \to 0^+$ produce $\n_i\n_j \log u \geq \frac{1}{2t} \delta_{ij}$, for which the equality is achieved when $u$ is the heat kernel on $\R^n$. 
\end{remark}


\begin{remark}
For proving the unique continuation result in Corollary \ref{Corollary unique continuation}, it is important to obtain a lower bound for $\n_i\n_j \log u$ that is asymptotic to $\tfrac{1}{2t}g_{ij}$ as $t\to 0^+$. 
\end{remark}

Tracing \eqref{matrix estimate K geq 0} yields a gradient estimate. 
\begin{corollary}\label{cor trace}
Under the same assumptions as in Theorem \ref{thm matrix Harnack heat equation}, we have
\begin{equation}\label{eq trace estimate K geq 0}
    \frac{u_t}{u} -\frac{|\n u|^2}{u^2} +\frac{n\k}{1-e^{-2\k t}} =\Delta \log u +\frac{n\k}{1-e^{-2\k t}}  \geq 0.
\end{equation}
\end{corollary}


On a general compact manifold, Hamilton \cite{Hamilton93} proved that for any positive solution $u$ to \eqref{heat equation fixed metric}, there exist constants $B$ and $C$ depending only on the geometry of $M$ (in particular the diameter, the volume, and the curvature and covariant derivative of the Ricci curvature) such that $t^{n/2}u\leq B$ and 
\begin{equation}\label{eq Hamilton matrix general case}
\n_i \n_j \log u  +\frac{1}{2t}g_{ij} +C \left(1+\log\left(\frac{B}{t^{n/2}u}\right)\right)g_{ij} \geq 0.
\end{equation}
Here, we also establish such a result for a general compact Ricci flow. 

\begin{theorem}\label{thm matrix Harnack heat equation general case}
Let $(M^n,g(t))$, $t\in [0,T]$, be a compact solution to the Ricci flow.
Let $u:M^n \times [0,T] \to \R$ be a positive solution to the heat equation \eqref{heat equation}. 
Suppose the sectional curvatures of $(M^n,g(t))$ are bounded by $K$ for some $K>0$. 
Then
\begin{equation}\label{eq matrix Harnack general case}
        \n_i \n_j \log u + \left(\frac{1}{2 t} + \frac{1}{t} \beta(t, n, K) \right) g_{ij} \ge 0,
\end{equation}
where 
$$\beta(t, n, K)= 4 \sqrt{n K t}  +C_2 (K+1)t +C_1 \sqrt{K} \diam.$$ Here $C_1>0$ is a numerical constant, $C_2>0$ depends only on the dimension and the non-collapsing constant $v_0 = \inf \{|B(x, 1, g(0))|_{g(0)}: x \in M^n \}$, and 
$$\diam:= \sup_{t\in [0,T]} \diam(M^n,g(t)).$$
\end{theorem}

Compared to \eqref{eq Hamilton matrix general case} for static metrics, our estimate \eqref{eq matrix Harnack general case} does not depend on the covariant derivatives of Ricci curvature. We have also made the dependence of the constants on curvature and the diameter more explicit. 
The proof of Theorem \ref{thm matrix Harnack heat equation general case} is more involved than that of Theorem \ref{thm matrix Harnack heat equation}, and the key steps are to establish bounds for heat kernel under Ricci flow, $\n \log u$, and $\Delta \log u$. 
In contrast to \eqref{eq Hamilton matrix general case} having the term $\log B/u$ whose order is not clear, \eqref{eq matrix Harnack general case} implies a lower bound of $\n_i \n_j \log u$ that is asymptotic to $C/t$ as $t\to 0^+$, where $C=\frac{1}{2}+C_1\sqrt{K}\diam$. Therefore, \eqref{eq matrix Harnack general case} can be regarded as a sharp version of Hamilton's estimate \eqref{eq Hamilton matrix general case}, just like \cite[Theorem 1.1]{Zhang21} is a sharp version of the Li-Yau estimate \cite[Theorem 1.3]{LY86}. Finally, we remark that the $C/t$-type bound does not hold for noncompact manifolds in general, such as on $\mathbb{H}^3$, the three-dimensional hyperbolic space, as explained in \cite{Zhang21}. 

A similar argument yields an improvement of Hamilton's classical matrix Harnack inequality \eqref{eq Hamilton matrix general case} in the static case on compact manifolds. 

\begin{theorem}\label{thm improve Hamilton}
Let $(M^n,g)$ be a closed Riemannian manifold and let $u:M^n \times [0,T] \to \R$ be a positive solution to the heat equation \eqref{heat equation fixed metric}. Suppose that the sectional curvatures of $M$ are bounded by $K$ and $|\n \Ric|\leq L$, for some $K,L>0$. Then
\begin{eqnarray}
 \n_i \n_j \log u  +
 \left(\frac{1}{2t}+(2n-1)K +\frac{\sqrt{3}}{2}L^{\frac{2}{3}} +\frac{1}{2t}\gamma(t,n,K,L)\right)g_{ij} \geq 0
\end{eqnarray}
for all $(x,t)\in M\times (0,T)$, where
\begin{eqnarray*}
 \gamma(t,n,K,L) &:=& \sqrt{nKt(2+(n-1)Kt)}  +\sqrt{C_3(K+L^{\frac{2}{3}})t(1+Kt)(1+K+Kt)}\\
&& + \left(2K(2+(n-1)Kt)+\frac{3}{2}L^{\frac{2}{3}}(1+(n-1)Kt) \right)\Diam, 
\end{eqnarray*}
$C_3>0$ depends only on the dimension $n$, and $\Diam$ denotes the diameter of $(M^n,g)$. 
\end{theorem}

Notice that Hamilton's original inequality \eqref{eq Hamilton matrix general case} has the term $C\log \left(\frac{B}{t^{\frac n 2}u}\right)$, where $B$ and $C$ depend on the geometry of the manifold and $B$ is greater than $t^{\frac n 2}u$, which itself is an additional assumption. The constant $C$ is equal to zero only when $M$ has nonnegative sectional curvature and parallel Ricci curvature. Otherwise, for this $\log$ term, we do not have any definite control on the order $q$ of $-t^{-q}$ coming out of this term, for general positive solutions, making this lower bound less practical. In Theorem \ref{thm improve Hamilton}, we manage to replace this term with a $C/t$ term, with $C$ depending only on $K$, $L$, and $\Diam$, which is of the correct order for $t$.

Next, we prove a matrix Li-Yau-Hamilton estimate for positive solutions to the backward conjugate heat equation coupled with the Ricci flow. 

\begin{theorem}\label{thm matrix Harnack backward conjugate heat equation}
Let $(M^n,g(t))$, $t\in [0,T]$, be a solution to the Ricci flow with nonnegative complex sectional curvature and $\Ric \leq \k g$ for some $\k>0$.
Let $u:M^n \times [0,T] \to \R$ be a positive solution to the backward conjugate heat equation \eqref{backward conjugate heat equation} on $M\times [0,T]$.
Suppose $\eta:(0,T) \to (0,\infty)$ is a $C^1$ function satisfying the ordinary differential inequality
\begin{equation}\label{eq c(t) ODE}
    \eta' \leq 2\eta^2 -2\kappa \eta -\frac{\kappa}{t} 
\end{equation}
on $(0,T)$ and that $\eta(t) \to \infty$ as $t\to T$. 
Then
\begin{equation}\label{eq matrix estimate KC geq 0 eta}
R_{ij} -\n_i\n_j \log u -\eta g_{ij} \leq 0
\end{equation}
for all $(x,t) \in M \times (0,T)$, where $R_{ij}$ denotes the Ricci curvature. 
\end{theorem}

\begin{remark}
On ancient Ricci flows, we can get rid of the $\frac{\k}{t}$ term in \eqref{eq c(t) ODE} and 
prove the cleaner estimate 
\begin{equation*}\label{eq matrix estimate KC geq 0 ancient}
R_{ij} -\n_i\n_j \log u -\frac{\k}{1-e^{-2\k(T-t)}} g_{ij} \leq 0.
\end{equation*}
See Theorem \ref{thm matrix LYH ancient} for details. 
\end{remark}

We give an explicit choice of $\eta(t)$. 
\begin{corollary}\label{corollary matrix Harnack conjugate}
Under the same assumptions as in Theorem \ref{thm matrix Harnack backward conjugate heat equation}, we have 
\begin{equation}\label{eq matrix estimate KC geq 0}
R_{ij} -\n_i\n_j \log u -\left(\frac{\k}{1-e^{-2\k(T-t)}} +\sqrt{\frac{\k}{2t}} \right) g_{ij} \leq 0.
\end{equation}
\end{corollary}

One motivation for bounding $R_{ij} -\n_i\n_j \log u$ from above comes from the work of Baldauf and Kim \cite{BK22}. They defined a parabolic frequency under Ricci flow and proved its monotonicity with a correction factor that depends on the upper bound of $R_{ij} -\n_i\n_j \log K$, where $K$ is the fundamental solution to \eqref{backward conjugate heat equation}. As an application of \eqref{eq matrix estimate KC geq 0}, we give an explicit correction factor in Proposition \ref{prop PF conjugate} in the nonnegative complex sectional curvature case.  
In addition, we believe such matrix Li-Yau-Hamilton estimates are of their own interest and should be useful in other situations. 


We need to assume nonnegative complex sectional curvature in Theorem \ref{thm matrix Harnack backward conjugate heat equation} because the proof uses Brendle's generalization \cite{Brendle09} of Hamilton's Harnack inequality for the Ricci flow \cite{Hamilton93JDG}. This feature is shared by the proof of \eqref{eq Ni matrix estimate forward} in \cite{Ni07}, which makes use of H.D. Cao's Harnack estimate for the K\"ahler-Ricci flow \cite{Cao92}. We also note that it suffices to assume $(M,g(0))$ has bounded nonnegative complex sectional curvature, as nonnegative complex sectional curvature is preserved by Ricci flows with bounded curvature (see Brendle and Schoen \cite{BS09} and Ni and Wolfson \cite{NWolfson07}). 

Finally, we would like to mention that since the pioneer works of Li and Yau \cite{LY86}, Hamilton \cite{Hamilton93}, Perelman \cite{Perelman1}, and others, various gradient and Hessian estimates for positive solutions to heat-type equations, with either fixed or time-dependent metrics, have been established by many authors, including Guenther \cite{Guenther02}, Ni \cite{Ni04Entropy, Ni04Addenda}, Cao and Ni \cite{CN05}, Ni \cite{Ni07}, Kotschwar \cite{Kotschwar07}, the second author \cite{Zhang06}, Souplet and the second author \cite{SZ06}, Kuang and the second author \cite{KZ08}, Cao \cite{Cao08},  X. Cao and Hamilton \cite{CH09}, Liu \cite{Liu09}, Bailesteanu, X. Cao, and Pulemotov \cite{BCP10}, X. Cao and the second author \cite{CZ11}, Li and Xu \cite{LX11}, Han and the second author \cite{HZ16}, Zhu and the second author \cite{ZZ17, ZZ18}, Huang \cite{Huang21}, and Yu and Zhao \cite{YZ20}, just to name a few. Our matrix Li-Yau-Hamilton estimates are new additions to the literature.

\subsection{Parabolic Frequency}
The elliptic frequency
$$I_A(r)=\frac{r \int_{B_r(p)} |\n u|^2 dx }{\int_{\p B_r(p)} u^2 dA}$$
for a harmonic function $u$ on $\R^n$ was introduced by Almgren \cite{Almgren79}. He used its monotonicity to study the local regularity of (multiple-valued) harmonic functions and minimal surfaces. The monotonicity of $I_A(r)$ also played an important role in studying unique continuation properties of elliptic operators by Garafalo and Lin \cite{GL86, GL87} and in estimating the size of nodal sets of solutions to elliptic and parabolic equations by Lin \cite{Lin91}. When $\R^n$ is replaced by a Riemannian manifold, Garofalo and Lin \cite{GaLin86} proved that there exist constants $R_0$ and $\Lambda$, depending only on the Riemannian metric, such that $e^{\Lambda r}I_A(r)$ is monotone nondecreasing in $(0, R_0)$ (see also Mangoubi \cite[Theorem 2.2]{Mangoubi13}). Frequency monotonicity is also crucial in the work of Logunov \cite{Logunov18b, Logunov18a} in estimating the size of nodal sets for harmonic functions and eigenfunctions on manifolds. In addition, frequency functions also play a crucial role in studying the dimension of the space of harmonic functions of polynomial growth on complete noncompact manifolds; see Colding and Minicozzi \cite{CM97a, CM97b}, G. Xu \cite{Xu16}, J.Y. Wu and P. Wu \cite{WW23}, Mai and Ou \cite{MO22}, and the references therein. For more applications, we refer the reader to the books \cite{HL} and \cite{Zelditch08}.

Poon \cite{Poon96} introduced the parabolic frequency
\begin{equation*}
 I_P(t)=\frac{t \int_{\R^n} |\nabla u|^2(x,T-t) G(x,x_0,t)\ dx}{\int_{\R^n} u^2(x,T-t) G(x,x_0,t) \ dx},
\end{equation*}
where $u$ solves the heat equation on $\R^n \times [0,T]$ and $G(x,x_0,t)$ is the heat kernel with a pole at $(x_0,0)$. He proved that $I_P(t)$ is monotone nondecreasing and derived some unique continuation results out of it. The monotonicity of $I_P(t)$ remains valid when $\R^n$ is replaced by a complete Riemannian manifold with nonnegative sectional curvature and parallel Ricci curvature, as remarked by Poon \cite[page 530]{Poon96} and proved independently by Ni \cite{Ni15}. The curvature conditions are needed to use Hamilton's matrix estimate \eqref{eq Hamilton matrix estimate}. 

Without assuming the restrictive parallel Ricci condition, Wang and the first author \cite{LW19} showed that $t e^{\sqrt{t}}I_P(t)$ is monotone nondecreasing for a short period of time on compact manifolds with nonnegative sectional curvature, which also produces a unique continuation result. They also defined the parabolic frequency
\begin{equation*}
    I_{LW}(t)=\frac{t \int_M |\n v(x,t)|^2 R d\mu_{g(t)}}{\int_M v^2(x,t) R d\mu_{g(t)}},
\end{equation*}
where $v$ solves the backward heat equation $v_t+\Delta_{g(t)} v=0$ coupled with a two-dimensional Ricci flow with positive scalar curvature. Using that $R$ satisfies the forward heat equation \eqref{forward conjugate heat equation} and admits a matrix Li-Yau-Hamilton estimate due to Hamilton (see \cite[Proposition 10.20]{CLN}), they showed that $I_{LW}(t)$ is monotone nondecreasing. 

Colding and Minicozzi \cite{CM22} proved that the parabolic frequency 
\begin{equation*}
    I_{CM}(t) = \frac{\int_M |\n u(x,t)|^2 e^{-f}d\mu_g}{\int_M u^2(x,t)e^{-f}d\mu_g},
\end{equation*} 
where $f$ is a smooth function on a Riemannian manifold $M^n$ and $u: M^n\times [0,T] \to \R^N$ solves the weighted heat equation $u_t-\Delta_f u=0$, is monotone nonincreasing without any curvature assumptions. The special case $f\equiv 1$ and $M^n$ being a bounded domain in $\R^n$ was treated in \cite[pages 61-62]{EvansPDEbook2nd} to prove the backward uniqueness of the heat equation with specified boundary values. 
They also defined a parabolic frequency for shrinking gradient Ricci solitons and showed its monotonicity with no curvature restrictions. 

For a general Ricci flow, Baldauf and Kim \cite{BK22} defined the parabolic frequency 
\begin{equation*}
    I_{BK}(t):=\frac{(T-t) \int_M |\n u(x,t)|^2 K(x,x_0,t) d\mu_{g(t)}}{\int_M u^2(x,t) K(x,x_0,t) d\mu_{g(t)}},
\end{equation*}
where $K$ is the backward conjugate heat kernel with a pole at $(x_0,T)$ and $u$ solves the heat equation \eqref{heat equation}. They were able to show that $e^{\int \frac{1-k(t)}{T-t} dt} \cdot I_{BK}(t)$ is monotone nonincreasing, where $k(t)$ is a time-dependent function such that
$$\Ric-\n^2 \log K \leq \frac{k(t)}{2(T-t)}.$$
More recently, C. Li, Y. Li, and K. Xu \cite{LLX22} studied the monotonicity of $I_{BK}(t)$ and its generalizations under the Ricci flow and the Ricci-harmonic flow. They obtained monotonicity formulas with correction factors depending on the bounds of the Bakry-\'Emery Ricci curvature or the Ricci curvature. It is also worth mentioning that H.Y. Liu and P. Xu \cite{LX22} investigated the monotonicity of a parabolic frequency for weighted $p$-Laplacian heat equation with $p\geq 2$ on Riemannian manifolds and obtained generalizations of \cite{CM22}. Using \eqref{eq Ni matrix estimate forward}, they generalized a frequency monotonicity formula of Ni \cite{Ni15} on K\"ahler manifolds to the setting of K\"ahler-Ricci flow. 

In this paper, we define a parabolic frequency for solutions to the backward conjugate heat equation \eqref{backward conjugate heat equation} coupled with the Ricci flow and prove its monotonicity up to certain correction factors. 
We shall use $G(x,x_0,t)$, the heat kernel with a pole at $(x_0,0)$, as a weight and define the following quantities:
\begin{eqnarray*}\label{eq def I D S}
I(t) &=& \int_M u^2(x,t) G(x,x_0, t) d\mu_{g(t)},  \\ \nonumber
D(t) &=& \int_M |\nabla u(x,t)|^2 G(x, x_0,t)d\mu_{g(t)}, \\ \nonumber
S(t) &=& \int_M u^2(x,t)R(x,t)  G(x, x_0,t) d\mu_{g(t)}.
\end{eqnarray*}
The first two quantities are direct generalizations of the terms in Poon's parabolic frequency $I_P(t)$ in the static case. The third one is new due to the Ricci flow. 
A natural generalization of Poon's frequency $I_P(t)$ is 
\begin{equation}\label{eq def F(t)}
F(t) : =\frac{I'(t)}{I(t)}=\frac{2 D(t) + S(t)}{I(t)}. 
\end{equation}

In the nonnegative sectional curvature case, we prove that
\begin{theorem}\label{thm PF sec geq 0}
Let $(M^n,g(t))$, $t \in [0,T]$, be a complete Ricci flow and let $u: M^n\times [0,T] \to \R$ be a solution to the backward conjugate heat equation \eqref{backward conjugate heat equation}. 
Suppose that $(M^n,g(t))$ has nonnegative sectional curvature and $\Ric \leq \kappa g$ for some constant $\kappa>0$. Then
\begin{equation}\label{eq PF F monotone sec geq 0}
e^{(n-2)\k t}(1-e^{-2\k t})F(t),
\end{equation}
where $F(t)$ is defined in \eqref{eq def F(t)}, in monotone nondecreasing on $[0,T]$. 
\end{theorem}

We point out that Theorem \ref{thm PF sec geq 0} covers, by letting $\kappa \to 0^+$, Poon's frequency monotonicity on $\R^n$ in \cite{Poon96}. It also implies a unique continuation result (see Lin \cite{Lin90}, Poon \cite{Poon96}, and Vessella \cite{Vessella03} for unique continuation results for parabolic equations).
\begin{corollary}\label{Corollary unique continuation}
Let $(M^n,g(t))$, $t \in [0,T]$, be a complete Ricci flow with nonnegative sectional curvature and $\Ric \leq \k g$ for some $\k>0$. Suppose that a solution $u(x,t)$ of the backward conjugate heat equation \eqref{backward conjugate heat equation} on $M^n \times [0,T]$ vanishes of infinity order at $(x_0, t_0) \in M\times (0,T)$, in the sense that 
\begin{equation*}
|u(x,t)| \leq O\left(d^2_t(x,x_0) +|t-t_0|\right)^N 
\end{equation*}
for all positive integer $N$ and all $(x,t)$ near $(x_0,t_0)$. 
Then $u\equiv 0$ in $M\times [0,T]$. 
\end{corollary}

Under general compact Ricci flows, we prove that $F(t)$ is monotone up to an implicit correction factor. 
\begin{theorem}\label{thm PF general}
Let $(M^n,g(t))$, $t \in [0,T]$, be a compact solution to the Ricci flow with sectional curvatures bounded by $K$ for some $K>0$. Let $u: M^n \times [0,T] \to \R$ be a solution to the backward conjugate heat equation \eqref{backward conjugate heat equation}. 
Then, for any $T>0$, there is a power $p=p(T, n, K, v_0, \diam)>0$ such that
\begin{equation}
t^p (F(t) +Z_0),
\end{equation}
where $F(t)$ is defined in \eqref{eq def F(t)}, is monotone nondecreasing on $[0,T]$. Here $Z_0=Z_0(T, n, K, v_0, \diam)$ is any sufficiently large number.
\end{theorem}

Extra curvature terms arise due to the Ricci flow when proving Theorem \ref{thm PF general}. We handle them using some cancellation property and some Li-Yau estimates for the heat kernel under the Ricci flow, together with the matrix Harnack inequality in Theorem \ref{thm matrix Harnack heat equation general case}.
Besides the above-mentioned results, we also prove the monotonicity of a parabolic frequency without weight at the end of Section 5 assuming nonnegative Ricci curvature; see Theorem \ref{thm PF no weight}.

Finally, we apply Theorem \ref{thm matrix Harnack backward conjugate heat equation} to prove that
\begin{proposition}\label{prop PF conjugate}
Let $(M^n,g(t))$, $t\in [0,T]$, be a solution to the Ricci flow with nonnegative complex sectional curvature and $\Ric \leq \k g$ for some $\k>0$. Let $u:M^n \times [0,T] \to \R$ be a solution to the heat equation \eqref{heat equation} and let $w:M^n \times [0,T] \to (0,\infty)$ be a positive solution to the backward conjugate heat equation \eqref{backward conjugate heat equation}. Then the quantity
\begin{equation}\label{eq PF KC geq 0}
(e^{2\k(T-t)} -1) e^{-\sqrt{8\k t}} \cdot \frac{\int_M |\n u(x,t)|^2 w(x,t) d\mu_{g(t)}}{\int_M u^2(x,t) w(x,t) d\mu_{g(t)}}
\end{equation}
is monotone nonincreasing on $[0,T]$. 
\end{proposition}

As mentioned before, Baldauf and Kim \cite{BK22} proved the monotonicity of $e^{\int \frac{1-k(t)}{T-t} dt}I_{BK}(t)$ under Ricci flow, where $k(t)$ is a time-dependent function such that
$$\Ric-\n^2 \log K \leq \frac{k(t)}{2(T-t)}.$$
However, it is not clear whether such $k(t)$ exists in the complete noncompact case. In the compact case, the existence of $k(t)$ is shown by Huang \cite{Huang21} and it depends on $|Rm|$, $|\n Rm|$ and $|\n^2 R|$, but no explicit $k(t)$ is known. 
Theorem \ref{thm matrix Harnack backward conjugate heat equation} and Corollary \ref{corollary matrix Harnack conjugate} provide an explicit $k(t)$ in the nonnegative complex sectional curvature case. Proposition \ref{prop PF conjugate} then gives an explicit correction factor in the monotonicity of $I_{BK}(t)$ and it is also applicable to complete noncompact Ricci flows with bounded nonnegative complex sectional curvature. In addition, we also get a unique continuation result in this case (see Corollary \ref{corollary UC KC geq 0}).

\textbf{Note:} Throughout the paper, we assume either $M$ is compact or $M$ is complete with bounded curvature/geometry, and the functions satisfy certain growth conditions so that the integrals are finite and all integration by parts can be justified.

The rest of this article is organized as follows. In Section 2, we derive the evolution equation satisfied by the Hessian of $\log u$, where $u$ is a positive solution to heat-type equations. 
Section 3 deals with the nonnegative sectional case and proves Theorem \ref{thm matrix Harnack heat equation}. 
Section 4 gives the proof of Theorem \ref{thm matrix Harnack heat equation general case}.
Section 5 is devoted to studying the parabolic frequency and proving Theorem \ref{thm PF sec geq 0} and Theorem \ref{thm PF general}. 
In Section 6, we prove Theorem \ref{thm matrix Harnack backward conjugate heat equation} and Proposition \ref{prop PF conjugate}. 
In Section 7, we prove Theorem \ref{thm improve Hamilton}. 
\section{Evolution Equations}

Let $(M^n,g(t))$, $t\in [0,T]$, be a solution to the Ricci flow. 
Let $u:M^n\times [0,T] \to \R$ be a positive solution to the heat-type equation
\begin{equation}\label{general heat equation}
    (\p_t  -\varepsilon \Delta_{g(t)}) u=\delta Ru,
\end{equation}
where $\varepsilon$ and $\delta$ are real parameters. We are mainly interested in the heat equation corresponding to $\ve=1$ and $\delta =0$ and the backward conjugate heat equation corresponding to $\ve=-1$ and $\delta =1$, but the calculations in this section are valid for all $\ve, \delta \in \R$.

The main result of this section is the evolution equation satisfied by 
\begin{equation*}
    H_{ij} :=  \n_i \n_j \log u.
\end{equation*}

\begin{proposition}\label{prop evolution H_ij}
In the setting described above, we have
\begin{eqnarray}\label{eq evolution H_ij} 
&& (\p_t-\ve \Delta)H_{ij} \\ \nonumber
&=&  \delta \n_i\n_j R + 2 \ve \left( H^2_{ij}  +R_{ikjl}\n_k v \n_l v + \n_k H_{ij} \n_k v \right)
\\ \nonumber 
&& +\ve (2R_{ikjl}H_{kl}-R_{ik}H_{jk}-R_{jk}H_{ik}) \\
&& +(1-\ve) (\n_iR_{jk}+\n_jR_{ik}-\n_kR_{ij})\n_k v, \nonumber
\end{eqnarray}
where $H^2_{ij}:=H_{ik}H_{jk}$. 
\end{proposition}

We first prove a commutator formula for $\p_t-\ve \Delta_L$ and $\n_i \n_j$, where $\Delta_L$ denotes the Lichnerowicz Laplacian acting on symmetric two-tensors via
\begin{equation*}
    \Delta_L h_{ij} = \Delta h_{ij}+2R_{ikjl}h_{kl}-R_{ik}h_{jk}-R_{jk}h_{ik}.
\end{equation*}

\begin{lemma}
Under the Ricci flow, it holds that 
\begin{eqnarray}\label{eq commuting Hessian and Delta_L}
&& (\p_t-\ve \Delta_L)(\n_i \n_j f)  \\ &=&  \n_i \n_j  (\p_t  -\ve \Delta) f    + (1-\ve)(\n_iR_{jk}+\n_jR_{ik}-\n_kR_{ij}) \n_k f \nonumber
\end{eqnarray}
for any smooth function $f(x,t)$. 
\end{lemma}
\begin{proof}
The cases $\varepsilon=\pm 1$ are proved in \cite{CLN}, so we only do a slight generalization here. The time derivatives of the Christoffel symbols $\Gamma_{ij}^k$ under the Ricci flow are given by (see \cite[page 108]{CLN}) 
\begin{equation*}
    \p_t \Gamma_{ij}^k = -g^{kl}(\n_i R_{jl} +\n_j R_{il} -\n_l R_{ij}).
\end{equation*}
It follows that 
\begin{equation*}
\p_t(\n_i \n_j f) = \n_i \n_j (\p_t f) +(\n_i R_{jk} +\n_j R_{ik} -\n_l R_{ij})\n_k f.
\end{equation*}
Second, we have 
\begin{equation*}
\Delta_L \n_i \n_j f =\n_i \n_j \Delta f +(\n_iR_{il}+\n_jR_{il}-\n_lR_{ij})\n_l f.   
\end{equation*}
This can be seen by commuting covariant derivatives as follows
\begin{eqnarray*}
\n_i \n_j \Delta f &=& \n_i \n_j \n_k \n_k f \\
&=& \n_i \n_k \n_k \n_j f -\n_i(R_{jl}\n_l f ) \\
&=& \n_k \n_i \n_k \n_j f -R_{il}\n_l \n_j f +R_{ikjl}\n_k\n_l f \\
&& -\n_iR_{jl}\n_l f -R_{jl}\n_i\n_l f \\
&=& \n_k \n_k \n_i \n_j f +\n_k(R_{ikjl}\n_l f) -R_{il}\n_l \n_j f +R_{ikjl}\n_k\n_l f \\
&& -\n_iR_{jl}\n_l f -R_{jl}\n_i\n_l f \\
&=& \Delta \n_i\n_j f +2R_{ikjl}\n_k\n_l f +\n_l R_{ij} \n_l f -\n_j R_{il} \n_l f \\
&&  -R_{il}\n_l \n_j f -\n_iR_{jl}\n_l f -R_{jl}\n_i\n_l f \\
&=& \Delta_L \n_i \n_j f -(\n_iR_{jl}+\n_jR_{il}-\n_lR_{ij})\n_l f, 
\end{eqnarray*}
where we have used the contracted Bianchi identity
\begin{equation*}
\n_k R_{ikjl} = \n_l R_{ij} -\n_j R_{il}.
\end{equation*}
Combining the above two calculations, we obtain \eqref{eq commuting Hessian and Delta_L}.   
\end{proof}

We now prove Proposition \ref{prop evolution H_ij}. 
\begin{proof}[Proof of Proposition \ref{prop evolution H_ij}]
For convenience, we write $v=\log u$. One derives from \eqref{general heat equation} that $v$ satisfies the equation
$$(\p_t -\ve \Delta)v =\ve |\n v|^2+\delta R.$$ 
We compute that 
\begin{eqnarray*}
    \n_i\n_j |\n v|^2 &=& 2\n_i (\n_j\n_k v) \n_k v +2 \n_i \n_k v \n_j \n_k v \\
    &=& 2\n_k (\n_i\n_j v) \n_k v +2H^2_{ij}  +2R_{ikjl}\n_k v \n_l v \\
    &=& 2H^2_{ij}  +2R_{ikjl}\n_k v \n_l v +2 \n_k H_{ij} \n_k v,
\end{eqnarray*}
where $H^2_{ij}:=H_{ik}H_{jk}$. 
Applying the identity \eqref{eq commuting Hessian and Delta_L} to $f=v$ yields
\begin{eqnarray*}
&& (\p_t-\ve \Delta_L) H_{ij} \\
&=& \n_i\n_j (\ve |\n v|^2+\delta R) \\
&& + (1-\ve) (\n_iR_{jk}+\n_jR_{ik}-\n_kR_{ij})\n_k v \\
&=& \delta \n_i\n_j R + 2\ve \left(H^2_{ij} +R_{ikjl}\n_k v \n_l v +\n_k H_{ij} \n_k v \right)
\\
&&+(1-\ve) (\n_iR_{jk}+\n_jR_{ik}-\n_kR_{ij})\n_k v. 
\end{eqnarray*}
Finally, \eqref{eq evolution H_ij} follows from the above identity and
\begin{equation*}
\Delta_L H_{ij} = \Delta_L H_{ij} +2R_{ikjl}H_{kl}-R_{ik}H_{jk}-R_{jk}H_{ik}. 
\end{equation*}
The proof is complete. 
\end{proof}

\section{Matrix Harnack for the Heat Equation}
In this section, we prove Theorem \ref{thm matrix Harnack heat equation}. The proof of Theorem \ref{thm matrix Harnack heat equation} divides into two cases: the compact case and the complete noncompact case. 

\subsection{The compact case}
\begin{proof}[Proof of Theorem \ref{thm matrix Harnack heat equation}]
In the compact case, we use Hamilton's tensor maximum principle to prove Theorem \ref{thm matrix Harnack heat equation}.

Setting $\ve=1$ and $\delta =0$ in \eqref{eq evolution H_ij}, 
we get that $H_{ij}:=\n_i\n_j \log u$ satisfies
\begin{eqnarray}\label{eq evolution H_ij S3} 
(\p_t-\Delta)H_{ij} 
&=& 2H^2_{ij} +(2R_{ikjl}H_{kl}-R_{ik}H_{jk}-R_{jk}H_{ik}) \\ \nonumber 
&& +2R_{ikjl}\n_k v \n_l v +2\n_k H_{ij} \n_k v
\end{eqnarray}
where $H^2_{ij}:=H_{ik}H_{jk}$.
We write
\begin{equation}\label{eq c(t) def}
    c(t):=\frac{\kappa }{1-e^{-2\kappa t}}
\end{equation}
and define
\begin{equation*}
    Z_{ij}:=H_{ij}+c(t)g_{ij}.
\end{equation*}
Direct calculations using \eqref{eq evolution H_ij S3} 
and the identity
\begin{eqnarray*}
&& 2H^2_{ij} +2R_{ikjl}H_{kl}-R_{ik}H_{jk}-R_{jk}H_{ik} \\
&=& 2Z^2_{ij} -4cZ_{ij} +2c^2 g_{ij} +2R_{ikjl}Z_{kl}-R_{ik}Z_{jk}-R_{jk}Z_{ik}
\end{eqnarray*}
show that
\begin{eqnarray}\label{eq evolution Z_ij S3} 
(\p_t-\Delta)Z_{ij} 
&=& 2Z^2_{ij} -4cZ_{ij} +2R_{ikjl}Z_{kl}-R_{ik}Z_{jk}-R_{jk}Z_{ik}\\ \nonumber
&& +2R_{ikjl}\n_k v \n_l v +2\n_k Z_{ij} \n_k v \\ \nonumber
&& +(c'+2c^2-2\k c)g_{ij}+2 c (\kappa  g_{ij} -R_{ij}).
\end{eqnarray}
Noting that $2R_{ikjl}\n_k v \n_l v \geq 0$, $R_{ij} \leq \k g_{ij}$, and $c(t)$ solves the ODE
\begin{equation*}
  c'=-2c^2 +2\kappa c,
\end{equation*}
we derive from \eqref{eq evolution Z_ij S3} that
\begin{eqnarray}\label{eq evolution Z_{ij} sec geq 0}
(\p_t-\Delta)Z_{ij} 
&\geq & 2Z^2_{ij} -4cZ_{ij} +2R_{ikjl}Z_{kl} -R_{ik}Z_{jk}-R_{jk}Z_{ik} \\ \nonumber
&& + 2\n_k Z_{ij} \n_k v.
\end{eqnarray}
Since $M$ is compact and $c(t) \to \infty$ as $t\to 0^+$, we have $Z_{ij}\geq 0$ as $t\to 0^+$. Then the tensor maximum principle of Hamilton \cite{Hamilton86} implies that $Z_{ij}\geq 0$ for all $t\in [0,T]$, as it is clear that 
$$2Z^2_{ij} -4cZ_{ij} +2R_{ikjl}Z_{kl}-R_{ik}Z_{jk}-R_{jk}Z_{ik}$$
is nonnegative at a null-eigenvector of $Z_{ij}$. 
The proof is complete. 
\end{proof}

\subsection{The complete noncompact case}

Now we deal with the case that $(M^n, g(t))$, $t\in [0,T]$, is a complete noncompact Ricci flow with nonnegative sectional curvature and $\Ric \leq \k g$. 
We note that the uniqueness of solutions to the heat equation \eqref{heat equation} fails to be true on a complete noncompact manifold. In order to apply Hamilton's tensor maximum principle (see for instance \cite[Theorem 12.33]{Chowbookpart2} for a version on complete noncompact Ricci flows) to $Z_{ij}$, one needs to impose some growth condition on the function $u$ and its first and second derivatives. Using an idea in \cite{CN05} and \cite{Ni07}, we can, however, get away without assuming any growth conditions on $u$. The key is that we are working with a positive solution of the heat equation and we can make use of the Li-Yau estimate for $u$ under the Ricci flow (see \cite[Theorem 25.9 and Corollary 25.13]{Chowbookpart3}) to obtain required growth estimates at any positive time. Then, we can get integral bounds on the first and second derivatives of $u$ via integration by parts. Finally, we use the idea of working with the smallest eigenvalue of the symmetric two-tensor $tuZ_{ij}$ to use a maximum principle for scalar heat equation (see \cite[Theorem 12.22]{Chowbookpart2}), avoiding the tensor maximum principle which requires a more restrictive growth condition.

We first prove a growth estimate for $u$ on a slightly smaller time interval using the Li-Yau estimate and its resulting Harnack inequality. 

\begin{lemma}\label{lemma u growth}
Let $(M^n,g(t))$, $t\in [0,T]$, be a complete solution of the Ricci flow with $-\k g \leq \Ric \leq \k g$. Let $u$ be a positive solution to the heat equation \ref{heat equation} on $M\times [0,T]$. Fix $p\in M$. For any $\delta \in (0, T/3)$, there exist a positive constant $A_1$, depending on $n,\k, T, \delta$, and $u(p,T)$ such that 
\begin{equation}\label{eq growth u}
    u(x,t) \leq \exp{\left( A_1(d^2_0(x,p) +1)\right)}
\end{equation}
for all $x\in M$ and $t \in [\delta, T-\delta]$, where $d_0(\cdot,\cdot)$ is the distance function with respect to $g(0)$. 
\end{lemma}

\begin{proof}
Since $-\k g \leq \Ric \leq \k g$, we have 
$$e^{-2\k T} g(0) \leq g(t) \leq e^{2\k T} g(0)$$
for all $t\in [0,T]$. 
In the Li-Yau estimate stated in \cite[Theorem 25.9]{Chowbookpart3} and the Harnack inequality stated in \cite[Corollary 25.13]{Chowbookpart3}, we can take $\tilde{g}=g(0)$, $\tilde{C}_0=e^{2\k T}$, $Q=0$, and $\varepsilon=\tfrac{1}{3}$ and let $R\to \infty$. Then, we conclude that there exist positive constants $B_1=B_1(n,\k)$ and $B_2=B_2(\k, T)$, such that
\begin{equation}\label{eq classical Harnack}
    \frac{u(x_2,t_2)}{u(x_1,t_1)} \geq e^{-B_1(t_2-t_1)}\left(\frac{t_2}{t_1}\right)^{-n}\exp{\left(-B_2\frac{d^2_{0}(x_1,x_2)}{t_2-t_1}\right)}
\end{equation}
for any $x_1,x_2 \in M$ and $0<t_1 <t_2 \leq T$. Applying \eqref{eq classical Harnack} with $x=x_1$, $x_2=p$, $t_1=t$, and $t_2=T$, we get  
\begin{equation}
u(x,t) \leq u(p,T)e^{B_1(T-t)}\left(\frac{T}{t}\right)^{n}\exp{\left(B_2\frac{d^2_{0}(x,p)}{T-t}\right)}
\end{equation}
for all $(x,t)\in M\times (0,T)$. 
This implies that, for $(x,t) \in M \times [\delta, T-\delta]$, there exists a constant $A_1$ depending on $n,\k, T, \delta$, and $u(p,T)$ such that \eqref{eq growth u} holds. 
\end{proof}

Next, we obtain integral bounds for the first and second derivatives of $u$. 
\begin{lemma}
Let $(M^n,g(t))$ and $u$ be the same as in Lemma \ref{lemma u growth}. Assume further that $(M^n,g(t))$ has nonnegative scalar curvature. Then there exists $A_2 \geq A_1$ such that 
\begin{equation}\label{eq growth nabla u}
\int_\delta^{T-\delta} \int_M |\n u(x,t)|^2 \exp{\left(-A_2 (d^2_0(x,p)+1) \right)} dx dt < \infty
\end{equation}
and 
\begin{equation}\label{eq growth nabla square u}
\int_\delta^{T-\delta} \int_M  |\n^2 u(x,t)|^2 \exp{\left(-A_2 (d^2_0(x,p)+1) \right)} dx dt < \infty.
\end{equation}
\end{lemma}
\begin{proof}
We derive from $u_t=\Delta u$ that 
\begin{equation*}
    (\p_t-\Delta)u^2 =-2|\n u|^2.
\end{equation*}
Multiplying both sides by a cut-off function $\vp^2$ (independent of time) and integrating by parts yield
\begin{eqnarray*}
&& 2 \int_\delta^{T-\delta} \int_M \vp^2 |\n u|^2 dx dt \\
&=& -\int_\delta^{T-\delta} \int_M \vp^2 (\p_t -\Delta) u^2 dx dt\\
&\leq& \int_M \vp^2 u^2(x,\delta)  dx +4 \int_\delta^{T-\delta} \int_M \vp u |\n \vp| |\n u| dx dt \\
&\leq& \int_M \vp^2 u^2(x,\delta)  dx +4 \int_\delta^{T-\delta} \int_M |\n \vp|^2 u^2 dx dt \\
&& + \int_\delta^{T-\delta} \int_M \vp^2 |\n u| ^2 dx dt. 
\end{eqnarray*}
Now \eqref{eq growth nabla u} follows from \eqref{eq growth u}. Applying the same argument to 
\begin{equation*}
(\p_t-\Delta)|\n u|^2 =2 |\n^2 u|^2
\end{equation*}
produces \eqref{eq growth nabla square u}. 

\end{proof}

\begin{proof}[Proof of Theorem \ref{thm matrix Harnack heat equation}: the complete noncompact case]
Recall that
$$Z_{ij}:=H_{ij}+c(t)g_{ij},$$ 
where $c(t)=\frac{\k}{1-e^{-2\k t}}$, satisfies \eqref{eq evolution Z_{ij} sec geq 0} under our curvature assumptions. 
Define 
\begin{equation*}
    \widetilde{Z}_{ij} :=tuZ_{ij}.
\end{equation*}
Using \eqref{eq evolution Z_{ij} sec geq 0}, we derive that 
\begin{eqnarray}\label{eq evolution widetilde Z_ij}
(\p_t -\Delta) \widetilde{Z}_{ij} &=& \tfrac{1}{t}\widetilde{Z}_{ij} +\tfrac{2}{tu}\widetilde{Z}^2_{ij}-4c(t)\widetilde{Z}_{ij} \\ \nonumber
 && +2R_{ikjl}\widetilde{Z}_{kl}-R_{ik}\widetilde{Z}_{jk}-R_{jk}\widetilde{Z}_{ik}.
\end{eqnarray}
Next, let's consider the function $\a(x,t)$ on $M\times [0,T]$ defined by
\begin{equation*}
    \a(x,t) =\inf\{s \geq 0: \widetilde{Z}_{ij}(x,t)+sg_{ij}(x,t) \geq 0 \}.
\end{equation*}
In other words,
\begin{equation}\label{eq alpha def}
    \a(x,t)=\max\{ 0, -\l_1(x,t)\}
\end{equation}
on $M\times (0,T)$, where $\l_1(x,t)$ is the smallest eigenvalue of $\widetilde{Z}_{ij}$ at $(x,t)$. 
The key is to show that 
\begin{equation}\label{eq alpha PDE}
    (\p_t-\Delta)\a \leq 0
\end{equation}
holds in the following barrier sense: for any $(x,t)\in M \times (0, T)$, we can find a neighborhood $U \subset M \times (0, T)$ of $(x,t)$ and a smooth (lower barrier) function $\phi: U \to \R$ such that $\phi \leq \a$ on $U$, with equality at $(x,t)$, and
\begin{equation}\label{eq phi}
    (\p_t-\Delta)\phi \leq 0
\end{equation}
at $(x,t)$. 
Note that inequality \eqref{eq alpha PDE} holds also in the viscosity sense and in the sense of distributions by standard arguments (see \cite{MMU14} for an elliptic version). 

To prove \eqref{eq alpha PDE}, we consider $Y_{ij}(x,t):=\widetilde{Z}_{ij}(x,t)+\a(x,t)g_{ij}(x,t)$. Fix $(x,t)\in M\times (0,T)$. By the definition of $\a(x,t)$, we have $Y_{ij}\geq 0$ on $M\times [0,T]$ and there exists a unit vector $e_1 \in T_xM$ such that $Y(e_1,e_1)=0$. We extend $e_1$ to an orthonormal basis $\{e_i\}_{i=1}^n$ of $T_xM$ consisting of eigenvector of $\widetilde{Z}_{ij}$ such that $\widetilde{Z}(e_i)=\l_i e_i$ with $\l_1 \leq \cdots \leq \l_n$. Next, we extend $\{e_i\}_{i=1}^n$ smoothly in a neighborhood $U$ of $(x,t)$ by parallel translation along radial geodesics using $\n^{g(t)}$ and regard the resulting vector fields, still denoted by $\{e_i\}_{i=1}^n$, as stationary in time in the sense that $\p_t e_i=0$ for each $1\leq i \leq n$. 

If $\l_1(x,t)>0$, then $\a\equiv 0$ near $(x,t)$ and the barrier function $\phi \equiv 0$ will fulfill \eqref{eq phi}. If $\l_1(x,t)\leq 0$, then the function $\phi(x,t)=-\frac{\widetilde{Z}(e_1,e_1)}{g(t)(e_1,e_1)}$ is defined in $U$ and gives a lower barrier for $\a(x,t)$ in that neighborhood. Therefore, we get using \eqref{eq evolution widetilde Z_ij} that at $(x,t)$,
\begin{eqnarray*}
(\p_t-\Delta)\phi &=&  -(\p_t-\Delta) \frac{\widetilde{Z}(e_1,e_1)}{g(t)(e_1,e_1)}\\
&=& -\tfrac{1}{t}\l_1 -\tfrac{2}{tu}\l_1^2 + 4c(t)\l_1 - 2R_{1k1k}\l_k \\ 
&\leq & \tfrac{1}{t}\a -\tfrac{2}{tu}\a^2 -4c(t)\a + 2\k \a \\
&\leq & \a\left(\tfrac{1}{t}-4c(t)+2\k \right) \\
&\leq & 0,
\end{eqnarray*}
where we have used \eqref{eq alpha def}, the estimate
$$R_{1k1k}\l_k = R_{1k1k}(\l_k-\l_1) +R_{11}\l_1 \geq -\a R_{11} \geq -\a \k,$$
and the elementary inequality
$$\tfrac{1}{t}-4c(t)+2\k < 0 \text{ for } t>0.$$
Hence, we have proved $(\p_t-\Delta)\a \leq 0$ in the barrier sense.

Without loss of generality, we may assume $u\geq \varepsilon >0$. This is because once the estimate has been established for $u_\varepsilon:=u+\varepsilon$, one can then let $\varepsilon \to 0$ and get the estimate for any positive $u$. 
By shifting the time from $t$ to $t+\delta$, we have the growth bounds \eqref{eq growth u}, \eqref{eq growth nabla u}, and \eqref{eq growth nabla square u}, which implies that there exists $b>0$ such that 
\begin{equation}\label{eq integral bound}
\int_0^T \int_M \exp{\left(-b d^2_0(x,p) \right)} \a^2(x,t) dx dt < \infty,
\end{equation}
Since $\a(x,0)=0$ for all $x\in M$, we can use the maximum principle (see \cite[Theorem 12.22]{Chowbookpart2}) to conclude that $\a(x,t)\leq 0$ on $M\times [0,T]$. 

\end{proof}

\section{Matrix Harnack for the Heat Equation: the General Case}
In this section, we prove Theorem \ref{thm matrix Harnack heat equation general case}. Without the nonnegativity of sectional curvatures, we have to estimate the terms involving curvature and derivatives of $u$ and the proof becomes much more involved. Here we employ an idea that has has been used in \cite{FudanNotes}, \cite{YZ22}, and \cite{Zhang21}, namely we first prove the estimate for the heat kernel and then derive the estimate for any positive solution to the heat equation.

\begin{proof}[Proof of Theorem \ref{thm matrix Harnack heat equation general case}]
The proof is divided into three steps.

{\it Step 1.} We derive a partial differential inequality satisfied by the smallest eigenvalue of $Q_{ij}:=t H_{ij}$, where $H_{ij} = \n_i \n_j \log u$ as before. 

A straightforward computation using \eqref{eq evolution H_ij S3} shows that $Q_{ij}$ satisfies
\begin{eqnarray}\label{eqQij}
(\p_t-\Delta)Q_{ij} 
&=& \frac{1}{t}Q_{ij} +\frac{2}{t}Q^2_{ij} +2R_{ikjl}Q_{kl}-R_{ik}Q_{jk}-R_{jk}Q_{ik} \\ \nonumber
&& +2tR_{ikjl}\n_k v \n_l v +2\n_k Q_{ij} \n_k v. 
\end{eqnarray} Here and through out this section $v =\log u$.

Let $\lambda_1$ be the minimum negative eigenvalue of $Q_{ij}$ in $M^n \times [0, t_0]$ which is reached at the point $(x_0, t_0)$. Note that we are done with the proof if no such $\lambda_1$ exists. Our task is to find a lower bound for $\lambda_1$. Let $\eta$ be a unit eigenvector with respect to the metric $g(t_0)$ at $x_0$. Using parallel transport, we extend $\eta$ along geodesic rays starting from $x_0$ so that it becomes a parallel unit vector field in a neighborhood of $x_0$ with respect to $g(t_0)$. This vector field, still denoted by $\eta=\eta(x)$, is regarded as stationary in the time interval $[0, t_0]$.
Now consider the vector field 
\[
\xi=\xi(x, t) = \frac{\eta(x)}{\Vert \eta(x) \Vert_{g(x, t)}}
\]which is a unit one with respect to $g(t)$. In local coordinates, we write 
$\xi=(\xi_1, ..., \xi_n)$ and we also introduce the scalar function 
\[
\Lambda =\Lambda (x, t) = \xi_i Q_{ij} \xi_j =\xi^T (Q_{ij}) \xi.
\]Notice that $\Lambda$ is a smooth function defined in a neighborhood of $x_0$ on the time interval $[0, t_0]$ and reaches its minimum value $\lambda_1$ at the point $(x_0, t_0)$.
Using \eqref{eqQij}, we find that 
\begin{eqnarray*}
&&\xi_i [(\p_t-\Delta)Q_{ij} ] \xi_j\\
&=& \frac{1}{t}\xi_i Q_{ij} \xi_j +\frac{2}{t}\xi_i Q^2_{ij} \xi_j +2R_{ikjl}Q_{kl}\xi_i \xi_j-R_{ik}Q_{jk} \xi_i \xi_j-R_{jk}Q_{ik} \xi_i \xi_j  \\ 
&& +2tR_{ikjl}\n_k v \n_l v  \xi_i \xi_j + 2\n_k Q_{ij} \n_k v   \xi_i \xi_j.
\end{eqnarray*}
Recall that that at $t=t_0$, $\xi$ is a parallel vector field and 
\[
\partial_t (\xi_i Q_{ij} \xi_j)=
\partial_t \left(\frac{\eta_i Q_{ij} \eta_j}{g_{ij} \eta_i \eta_j}\right)
= \xi_i \partial_t Q_{ij} \xi_j + 2 \xi_i Q_{ij} \xi_j R_{kl}
\xi_k \xi_l.
\]Combining the above two identities, we deduce, at $(x,t)=(x_0, t_0)$, that
\begin{eqnarray}\label{eqlamb}
(\p_t-\Delta) \Lambda
&=& \frac{1}{t} \Lambda +\frac{2}{t} \Lambda^2 +2R_{ikjl}Q_{kl}\xi_i \xi_j  \\ \nonumber
&& +2tR_{ikjl}\n_k v \n_l v  \xi_i \xi_j + 2\n_k \Lambda \n_k v.
\end{eqnarray} 
Notice the terms involving the Ricci curvature are canceled. We remark that this equation may not be satisfied for $t<t_0$ but the proof uses this equation only at $(x,t)=(x_0, t_0)$. As mentioned, $\Lambda$ reaches its minimum value at $(x_0, t_0)$. Therefore, \eqref{eqlamb} implies, at $(x_0, t_0)$
\begin{eqnarray}
\label{ineqlamb1}
2 \lambda^2_1 \le - \lambda_1  - 2 t R_{ikjl}Q_{kl}\xi_i \xi_j  
- 2t^2 R_{ikjl}\n_k v \n_l v  \xi_i \xi_j.
\end{eqnarray}

Next, we aim to bound the curvature terms on the right-hand side of \eqref{ineqlamb1}. First, we write
\begin{eqnarray*}
 2R_{ikjl}Q_{kl} &=& 2(R_{ikjl}+K(g_{ij}g_{kl}-g_{il}g_{jk}))Q_{kl} -2K(g_{ij}g_{kl}-g_{il}g_{jk})Q_{kl}\\
 &:=& 2 w_{ikjl} Q_{kl} -2K(g_{ij}g_{kl}-g_{il}g_{jk})Q_{kl}.
\end{eqnarray*}
Besides the lowest negative eigenvalue $\lambda_1$, let $\lambda_2$,  ..., $\lambda_n$ be other eigenvalues of $(Q_{ij})$ at $(x_0, t_0)$ arranged in increasing order. After diagonalizing $(Q_{ij})$ at $(x_0, t_0)$ with an orthonormal basis $\{\xi, ... \}$, we deduce 
\begin{eqnarray*}
 R_{ikjl}Q_{kl} \xi_i \xi_j 
 &=& \sum_k w_{1k1k} \lambda_k -K(g_{ij}g_{kl}-g_{il}g_{jk})Q_{kl} \xi_i \xi_j \\
 &=&\sum_{\lambda_k \ge 0} w_{1k1k} \lambda_k + \sum_{\lambda_k < 0} \,  w_{1k1k} \lambda_k
 - K t \Delta v + K \lambda_1\\
 &\geq& \sum_{\lambda_k < 0} \,  w_{1k1k} \lambda_1 -K t \Delta v +K \lambda_1.
\end{eqnarray*}
Here we just used the assumption on sectional curvature 
$$
R_{ikjl}\geq -K(g_{ij}g_{kl}-g_{il}g_{jk}) 
$$ or $w_{ikjl}\geq 0$ and the identity 
\[
\tr (Q_{ij})= t \Delta  v.
\] 
Using the upper bound of the sectional curvature 
$$
R_{ikjl}\le K(g_{ij}g_{kl}-g_{il}g_{jk}), 
$$ we see that
\[
\sum_k w_{1k1k} \le 2K \sum_k (g_{11}g_{kk}-g_{1k}g_{1k})= 2K (n-1)
\]and we arrive at, via  $\lambda_1<0$, that
\begin{equation}
\label{RQXX<}
R_{ikjl}Q_{kl} \xi_i \xi_j \ge (2n-1) K \lambda_1  -K t \Delta v.
\end{equation}

Using the lower bound on the sectional curvatures again, noticing $\xi=(1, 0, ..., 0)$ and $g_{ij} =\delta_{ij}$ at $(x_0, t_0)$ by our choice of the orthonormal coordinates,  we have that 
\begin{eqnarray}\label{Rvv<}
R_{ikjl}\n_k v \n_l v \xi_i \xi_j &=& R_{1k1l} \n_k v \n_l v \\ \nonumber
&\geq&  -K (g_{kl}-g_{1l}g_{1k}) \n_k v \n_l v \\ \nonumber
&=& -K|\n v|^2 + K |\n_1 v|^2 \\ \nonumber
&\ge&  -K|\n v|^2.
\end{eqnarray} 
Substituting \eqref{RQXX<}, \eqref{Rvv<} into \eqref{ineqlamb1}, we deduce, for $v = \log u$,
\begin{eqnarray}
\label{ineqlamb2}
2 \lambda^2_1 \le - \lambda_1  - 2 t (2n-1) K  \lambda_1  + 2 K t^2 \Delta v 
+ 2 t^2 K|\n v|^2.
\end{eqnarray}

{\it Step 2.}  We need to bound the right hand side of \eqref{ineqlamb2}. This might be difficult for all positive solutions $u$ but doable for the heat kernel.

In this step, we assume $u(x,t)=G(x,t,y):=G(x, t; y, 0)$ is the heat kernel with a pole at $y\in M, t=0$ and $v=\log u$. It is known that the following curvature-free bound holds:
\begin{equation}
\label{tt0dlnu}
    (t-(t_0/2))^2 |\n \log u|^2 \leq (t-(t_0/2)) \log \frac{\sup\{u(x,t) : (x,t) \in M \times [t_0/2,t_0] \}}{u}
\end{equation}
for all $(x,t) \in M \times [t_0/2,t_0] $ (see \cite{Zhang06} or \cite{CH09}). Under the condition of bounded sectional curvature, the upper and lower bound for the heat kernel can be obtained in a classical way in any finite time interval and hence are more or less known. By now, we know that only the pointwise bound on the scalar curvature and initial volume non-collapsing condition are needed for the heat kernel bounds to hold (see \cite[Theorem 1.4]{BZ17}). Using that theorem repeatedly over fixed time intervals and taking advantage of the reproducing formula of the heat kernel, we know that the following bounds hold: there exists a numerical positive constant $C_1$ and another positive constant $C_2$ depending only on the volume non-collapsing constant of $g_0$ and the dimension such that
\begin{align}\label{HKbound}
\frac{1}{C_2 t^{n/2}} e^{ -C_2 K t- C_1 d^2(x, y, t)/t} \le G(x, t, y) \le \frac{C_2}{t^{n/2}} e^{ C_2 K t-d^2(x, y, t)/(C_1 t)}.
\end{align} 
Alternatively, since the sectional curvature is bounded, one can just follow the classical method by Li-Yau to obtain such bounds.
The above bounds are far from optimal for large times.  Since the manifold is compact, the large-time behavior of the heat kernel is relatively simple since positive solutions tend to be constant. But we will not pursue an optimal large-time bound this time.

Substituting \eqref{HKbound} to \eqref{tt0dlnu}, we obtain, for all $t_0>0$
\begin{equation}\label{tt0dlnu2}
    t^2_0 |\n \log u|^2(x, t_0) \leq C_3 (K+1) t^2_0+ 2 C_1 \sup_{t \in [t_0/2, t_0]} d^2(x, t, y).
\end{equation} Here $C_3$ depends only on the volume noncollapsing constant of $g_0$ and the dimension.

Next, we need to find an upper bound for the term $t^2 \Delta \log u= t^2(\frac{\Delta u}{u} - 
\frac{|\n u|^2}{u^2})$. In the stationary case, this is done in Hamilton \cite[Lemma 4.1]{Hamilton93}. Following that proof, the Ricci flow produces one extra term involving the Ricci curvature. Since the sectional curvature is bounded, we can treat this term without much difficulty.

Let $L$ be the operator
\begin{equation}
L = \Delta + 2 \n \log u \n - \partial_t.
\end{equation} The following identities are well known and also follow from the calculations in Section 2. 
\begin{equation}
\label{Lddu/ulnu}
L (\Delta u/u) = 2 R_{ij} \n_i \n_j u / u, \quad L (|\n \log u|^2) = 2|\n_i \n_j \log u |^2.
\end{equation} 
Therefore,
\begin{eqnarray*}
 && L \left(\frac{\Delta u}{u} +  |\n \log u|^2 \right) \\
 &=&   2|\n_i \n_j \log u |^2 - 2 R_{ij} \left( \frac{\n_i\n_j u}{u} - 
 \frac{\n_i u \n_j u}{u^2} \right) - 2 R_{ij} \frac{\n_i u \n_j u}{u^2}\\
 &=& 2|\n_i \n_j \log u |^2 - 2 R_{ij} \n_i \n_j \log u - 2 R_{ij} \frac{\n_i u \n_j u}{u^2}\\
 &=&|\n_i \n_j \log u |^2 + | \n_i \n_j \log u - R_{ij} |^2 - R^2_{ij}- 2 R_{ij} \frac{\n_i u \n_j u}{u^2}\\
 & \ge& \frac{1}{n} |\Delta \log u|^2 + \frac{1}{n} |\Delta \log u -R|^2 - C_n K^2 - C_n K |\n \log u|^2.
\end{eqnarray*} Here, $C_n$ is a dimensional constant and the assumption $|R_{ijkl}| \le K (g_{ik}g_{jl}-g_{il}g_{jk})$ has been used.  Writing
\[
Y= \frac{\Delta u}{u} +  |\n \log u|^2,
\]
then the above inequality implies
\begin{eqnarray*}
 L Y &\ge&   \frac{1}{n} \left|Y- 2|\n \log u|^2 \right|^2
+ \frac{1}{n} \left|Y- 2|\n \log u|^2 -R \right|^2 \\
&& - C_n K^2 - C_n K |\n \log u|^2.  
\end{eqnarray*}
Therefore,
\begin{eqnarray}\label{Lt2Y>}
 L (t^2 Y) &\ge&   \frac{1}{n t^2} \left|t^2 Y- 2 t^2|\n \log u|^2 \right|^2 \\ \nonumber
 && + \frac{1}{n t^2} \left|t^2 Y- 2 t^2|\n \log u|^2 - t^2 R \right|^2 \\ \nonumber
&& - C_n t^2 K^2 - C_n K t^2 |\n \log u|^2 - 2 \frac{t^2 Y}{t}.   
\end{eqnarray}

Let $T>0$ be any fixed time and $(x_0, t_0)$ be a  maximal point of $t^2 Y$ in $M^n \times (0, T]$ where $t^2 Y$ reaches a positive maximum value. Note that $t_0$ may be less than $T$ but the argument by maximum principle together with \eqref{tt0dlnu2}  is good enough to bound $t^2 Y$ up to $T$ and hence for all time. By \eqref{Lt2Y>}, we know, at $(x_0, t_0)$, the following inequality holds:
\[
\frac{1}{n t^2} \left|t^2 Y- 2 t^2|\n \log u|^2 \right|^2 
\le  C_n t^2 K^2 + C_n K t^2 |\n \log u|^2 + 2 \frac{t^2 Y}{t}.   
\]Hence for all $t \in (0, T]$, we have 
\[
\frac{1}{2 n t^2} |t^2 Y|^2 \le  \left(\frac{1}{n t^2}  |2 t^2|\n \log u|^2 |^2 
+ C_n t^2 K^2 + C_n K t^2 |\n \log u|^2 + 2 \frac{t^2 Y}{t} \right)_{(x_0, t_0)}.   
\]
Now we take $u$ to be the heat kernel $G(x, t, 0)$. Using \eqref{tt0dlnu2} we conclude that
\[
t^2 Y \le 4n t +  C_3 (K+1) t^2 + 2 C_1  \diam^2
\]which yields, since $Y= \frac{\Delta u}{u} +  |\n \log u|^2$, that
\begin{equation}
\label{t2ddu/ujie}
 t^2 \frac{\Delta u}{u} \le 4nt  +  C_3 (K+1) t^2 + 2 C_1  \diam^2.
\end{equation}

From \eqref{ineqlamb2} and the relation $\Delta v = \Delta \log u =\frac{\Delta u}{u} -  |\n \log u|^2 $, we see that
\begin{eqnarray}
\label{ineqlamb22}
2 \lambda^2_1 \le - \lambda_1  - 2 t (2n-1) K  \lambda_1  + 2 K t^2 \frac{\Delta u}{u}.
\end{eqnarray} 
Using \eqref{t2ddu/ujie}, we deduce that
\begin{eqnarray}\label{ineqlamb23}
2 \lambda^2_1 &\le&  - \lambda_1  - 2 t (2n-1) K  \lambda_1  \\ \nonumber
&& + 2 K [4nt  +  C_3 (K+1) t^2 + 2 C_1 \diam^2].
\end{eqnarray}
Since $\lambda_1<0$ by assumption, we conclude, after elementary estimates, 
\begin{eqnarray}
\label{ineqlamb24}
t (H_{ij})& \ge&  \lambda_1 (g_{ij}) \\ \nonumber
& \ge& \left(- \frac{1}{2} -4 \sqrt{nK t} 
   -C_2 (K+1)t -C_1 \sqrt{K} \diam \right) (g_{ij})\\ \nonumber
   &:=&  \left(- \tfrac{1}{2} -\beta(t, n, K)\right) (g_{ij}),
\end{eqnarray} 
where $C_1$ is a numerical constant and $C_2$ depends only on the non-collapsing constant $v_0$ of $g_0$ and the dimension. Note that we have renamed $C_3$ to $C_2$ for consistency with the statement of the theorem. 
\medskip

{\it Step 3.} Finally, we show the matrix Harnack estimate holds for any positive solution $u(x,t)$ to the heat equation. 

Note that
\begin{equation*}
u(x,t) =\int_M G(x,t,y)u_0(y) dy.
\end{equation*}
Differentiating under the integral yields
\begin{eqnarray*}
\n_j u(x, t) &=& \int_M \n_j G(x,t,y) u_0(y) dy, \\
\n_i \n_j u(x, t) &=& \int_M \n_i \n_j G(x,t,y) u_0(y) dy. 
\end{eqnarray*} 
Here and later in the step $dy=dg(0)(y)$ etc.
Therefore, 
\begin{eqnarray*}
  && t u^2 \n_i\n_j \log u(x, t) \\
  &=&t ( u \n_i \n_j u -\n_i u \n_j u)(x, t) \\
  &=& \int_M t G(x,t,z) u_0(z) dz \int_M \n_i \n_j G(x,t,y) u_0(y) dy \\
  && -\int_M t \n_j G(x,t,z) u_0(z) dz \int_M \n_j G(x,t,y) u_0(y) dy.
\end{eqnarray*} 
Hence,
\begin{eqnarray}\label{tu2uij}
&&t u^2 \n_i\n_j \log u(x, t) \\ \nonumber
  &=& \int_M \int_M t G(x,t,z)  \n_i \n_j G(x,t,y) u_0(z) u_0(y) dzdy \\ \nonumber
  && -\int_M \int_M  t \n_j G(x,t,z) \n_j G(x,t,y)  u_0(z) u_0(y) dzdy.
\end{eqnarray}  

Fixing a space time point $(x, t)$, $t>0$. Let us diagonalize $(Q_{ij})=t (\n_i\n_j\log u(x, t))$
using its orthonormal eigenvectors $\{e_1, ..., e_n \}$ such that $e_1$ corresponds to the smallest eigenvalue $\lambda_1$. By \eqref{tu2uij}, we have
\begin{eqnarray}\label{tu2u11}
&&t u^2 \n_1\n_1 \log u(x, t) \\ \nonumber
  &=& \int_M \int_M t G(x,t,z)  \n_1 \n_1 G(x,t,y) u_0(z) u_0(y) dzdy \\ \nonumber
  && -\int_M \int_M  t \n_1 G(x,t,z) \n_1 G(x,t,y)  u_0(z) u_0(y) dzdy.
\end{eqnarray} 
According to \eqref{ineqlamb24} in Step 2, the following holds
\begin{eqnarray*}
t \left( \frac{\n_1 \n_1 G(x,t,y)}{G(x, t, y)} - \frac{|\n_1 G(x,t,y)|^2}{G^{2}(x, t, y)}\right) \ge 
-\frac{1}{2}-\beta(t, n, K)
\end{eqnarray*}
so that
\begin{eqnarray*}
t \n_1 \n_1 G(x,t,y)  &\ge& t \frac{|\n_1 G(x,t,y)|^2}{G(x, t, y)}  \\
&& -\left(\tfrac{1}{2}+\beta(t, n, K)\right)  G(x, t, y). 
\end{eqnarray*}
Substituting the last inequality into \eqref{tu2u11} and regrouping the third term on the right-hand side, we deduce
\begin{eqnarray}
\label{tu2u112}
&& t u^2 \n_1\n_1 \log u(x, t) \\ \nonumber
  &\ge& \int_M \int_M  G(x,t,z) t \frac{|\n_1 G(x,t,y)|^2}{G(x, t, y)}  u_0(z) u_0(y) dzdy \\ \nonumber
  && - [\frac{1}{2}+\beta(t, n, K)] \int_M \int_M   G(x,t,z) G(x,t,y)  u_0(z) u_0(y) dzdy\\ \nonumber 
  && -\int_M \int_M \frac{  \sqrt{t} \n_1 G(x,t,z) \sqrt{G(x,t,y)}}{ \sqrt{G(x,t,z)}} 
  \sqrt{u_0(z) u_0(y)} \\ \nonumber 
  && \times \frac{  \sqrt{t} \n_1 G(x,t,y) \sqrt{G(x,t,z)}}{ \sqrt{G(x,t,y)}} 
  \sqrt{u_0(z) u_0(y)}  dzdy.
\end{eqnarray}
Observe that the first term on the right-hand side dominates the third term due to the Cauchy-Schwarz inequality and the integral in the second term is $u^2(x,t)$. Hence we have proven
\[
t \n_1\n_1 \log u(x, t) \ge - \frac{1}{2} - \beta(t, n, K).
\]
Since the left-hand side is the smallest eigenvalue of $(Q_{ij})$, the proof is done.
\end{proof}

\section{Parabolic Frequency Monotonicity}

\subsection{Parabolic Frequency}
Let $(M^n,g(t))$, $t \in [0,T]$, be a complete solution to the Ricci flow. 
Let $u$ be a solution to the (backward) conjugate heat equation $(\p_t+\Delta)u=Ru$. 
Let $G(x,x_0,t)$ be the heat kernel of the heat equation \eqref{heat equation} with the pole at $(x_0, 0)$.  We defined in the Introduction section that
\begin{eqnarray}
I(t) &:=& \int_M u^2(x,t) G(x,x_0, t) dg(t)(x), \label{eq def I(t)} \\
D(t) &:=& \int_M |\nabla u(x,t)|^2 G(x, x_0,t) dg(t)(x), \label{eq def D(t)} \\
S(t) &:=& \int_M u^2(x,t)R(x,t)  G(x, x_0,t) dg(t)(x) \label{eq def S(t)}.
\end{eqnarray} 
Here $dg(t)(x):=d\mu_{g(t)}$ is the Riemannian measure induced by $g(t)$. 

Below we will simply write $I(t)=\int_M u^2G dg$ and similar notations for other integrals if no confusion arises.

\begin{lemma} $I(t)$ defined in \eqref{eq def I(t)} satisfies
\begin{equation}\label{eq I derivative}
    I'(t) = 2D(t)+S(t). 
\end{equation}

\end{lemma}

\begin{proof} 
Under the Ricci flow, the measure $d{g(t)}$ evolves by
$\p_t (d{g(t)})=-R d {g(t)}.$ 
A straightforward computation shows that
\begin{eqnarray*}
I'(t) &=& \int_M 2uu_t G dg +\int_M u^2 G_t dg-\int_M u^2 R G dg\\
&=& \int_M 2u(Ru-\Delta u )G dg +\int_M u^2 \Delta G dg-\int_M u^2 R G dg\\
&=& \int_M u^2 RG dg-\int_M (\Delta u^2- 2|\nabla u|^2) G dg+ \int_M \Delta u^2 G dg\\
&=& 2\int_M |\nabla u|^2 G dg+\int_M u^2 RG dg\\
&=& 2D(t)+S(t). 
\end{eqnarray*}
\end{proof}

\begin{lemma}
$D(t)$ defined in \eqref{eq def D(t)} satisfies
\begin{eqnarray}\label{eq D'(t)}
D'(t)
&=& 2 \int_M (\Ric-\n^2 f) (\n u, \n u) G dg+2 \int_M (\Delta_f u)^2 G dg \\ \nonumber
&& -2\int_M Ru(\Delta_f u) G  dg -\int_M |\n u|^2 R G dg, 
\end{eqnarray}
where $f=-\log G$ and $\Delta_f := \Delta -\langle \n f, \n \cdot \rangle$. 
\end{lemma}
\begin{proof}
Noticing $\p_t|\n u|^2 =2\Ric(\n u, \n u) +2 \langle \n u, \n u_t\rangle$ and $\p_t dg(t)=-Rdg(t)$, we get by differentiating under the integral that
\begin{eqnarray}\label{eq D'(t) 2}
D'(t) &=& 2\int_M \Ric(\n u, \n u) Gdg  +2\int_M \langle\n u, \n u_t \rangle G dg \\ \nonumber
&& +\int_M |\n u|^2 G_t dg-\int_M |\n u|^2 R G dg\\ \nonumber
&=& 2\int_M \Ric(\n u, \n u) G dg +2\int_M \langle\n u, \n (Ru -\Delta u) \rangle G dg\\ \nonumber
&& +\int_M |\n u|^2 \Delta G dg-\int_M |\n u|^2 R G dg\\ \nonumber
&=& 2\int_M \Ric(\n u, \n u) G dg+2\int_M \langle\n u, \n (Ru) \rangle G dg\\ \nonumber 
&& -2\int_M \langle\n u, \n (\Delta u) \rangle G dg +\int_M \Delta |\n u|^2 G dg \\ \nonumber
&& -\int_M |\n u|^2 R Gdg.
\end{eqnarray}
In view of the integration by parts 
\begin{equation*}
    \int_M \langle\n u, \n (Ru) \rangle G dg= -\int_M Ru \Delta_f u Gdg  
\end{equation*} 
and the Bochner formula 
\begin{equation*}
    \Delta |\n u|^2 =2|\n^2 u|^2 +2\langle \n u, \n \Delta u \rangle +2\Ric(\n u, \n u), 
\end{equation*}
we derive from \eqref{eq D'(t) 2} that
\begin{eqnarray}\label{eq D'(t) 3}
D'(t) &=& 4 \int_M \Ric(\n u, \n u) G dg + 2\int_M |\n^2 u|^2 G dg\\ \nonumber 
&& +2\int_M Ru \Delta_f u G dg-\int_M |\n u|^2 R G dg, 
\end{eqnarray}
The weighted Bochner formula 
\begin{equation*}
    \Delta_f |\n u|^2 =2|\n^2 u|^2 +2\langle \n u, \n \Delta_f u \rangle +2\Ric(\n u, \n u) +2\n^2 f(\n u, \n u),
\end{equation*}
implies that 
\begin{eqnarray*}
&& 2\int_M \Ric(\n u, \n u)G +2\int_M |\n^2 u|^2 G dg\\
&=& -2 \int_M \langle \n u, \n \Delta_f u \rangle G dg-2 \int_M \n^2 f(\n u, \n u) dg\\
&=& 2 \int_M (\Delta_f u)^2 G dg-2 \int_M \n^2 f(\n u, \n u) dg. 
\end{eqnarray*}
Substituting the above identity into \eqref{eq D'(t) 3} produces
\begin{eqnarray*}
D'(t)
&=& 2 \int_M (\Ric-\n^2 f) (\n u, \n u) Gdg  +2 \int_M (\Delta_f u)^2 G dg \\
&& -2\int_M Ru(\Delta_f u) G dg -\int_M |\n u|^2 R G dg.
\end{eqnarray*}
This proves \eqref{eq D'(t)}. 
\end{proof}

\begin{lemma}
$S(t)$ defined in \eqref{eq def S(t)} satisfies
\begin{eqnarray}\label{eq S'(t)}
S'(t)  &=& \int_M u^2 R^2 G dg  +2\int_M u^2 |\Ric|^2 G dg+2\int_M |\n u|^2 RG dg\\ \nonumber 
    && +2\int_M u^2 R\Delta G dg +4\int_M uR\langle \n u, \n G \rangle dg.  
\end{eqnarray}
\end{lemma}

\begin{proof}
We compute using $\p_t R =\Delta R+2|\Ric|^2$ that
\begin{eqnarray*}
    S'(t) &=& \int_M 2uu_tRG dg+\int u^2 R_t G dg\\
    && +\int_M u^2 R G_t dg -\int u^2R^2 G dg\\
    &=& \int_M 2u(Ru-\Delta u)RG dg+\int u^2(\Delta R+2|\Ric|^2)G dg\\
    && +\int_M u^2 R\Delta G dg-\int u^2R^2 G dg\\
    &=& \int_M u^2 R^2 G dg- 2\int_M u\Delta u RG dg +\int u^2 \Delta R Gdg  \\
    && +2 \int_M u^2 |\Ric|^2 G dg+\int_M u^2 R\Delta G dg.
\end{eqnarray*}
Observing 
\begin{eqnarray*}
\int_M u^2 \Delta R G dg &=& \int_M R \Delta (u^2 G) dg\\
&=& \int_M  R\left( \Delta u^2 G +u^2 \Delta G +2\langle \n u^2, \n G \rangle \right)dg  \\
&=& 2\int_M u\Delta u RG dg+2\int_M |\n u|^2 RG dg\\
&& +\int_M u^2R \Delta G dg+ 4\int_M uR\langle \n u, \n G \rangle dg, 
\end{eqnarray*}
we deduce that
\begin{eqnarray*}
S'(t)  &=& \int_M u^2 R^2 G dg+2\int_M u^2 |\Ric|^2 G +2\int_M |\n u|^2 RG dg\\
    && +2\int_M u^2 R\Delta G dg+ 4\int_M uR\langle \n u, \n G \rangle dg.  
\end{eqnarray*}
\end{proof}

\begin{lemma}\label{lemI''t}
For $I(t)$ defined in \eqref{eq def I(t)}, we have 
\begin{eqnarray}\label{eq I''(t)}
I''(t) &=& 4 \int_M (\Ric-\n^2 f)(\n u, \n u) G dg \\ \nonumber 
&& +\int_M \left(2\Delta_f u -Ru \right)^2 G dg+2\int_M u^2 |\Ric|^2 G dg \\ \nonumber 
&& +2\int_M u^2 R\Delta G dg+4\int_M Ru\langle \n u, \n G \rangle dg. 
\end{eqnarray}
\end{lemma}

\begin{proof}
By \eqref{eq I derivative}, we have $I'(t)=2D(t)+S(t)$. 
Using \eqref{eq D'(t)} and \eqref{eq S'(t)}, we calculate
\begin{eqnarray*}
I''(t) &=& 2D'(t) +S'(t) \\
&=& 4 \int_M (\Ric-\n^2 f)(\n u, \n u) G dg+4 \int_M (\Delta_f u)^2 G dg \\
&& -4\int_M Ru(\Delta_f u) G dg -2\int_M |\n u|^2 R G dg \\
&&+ \int_M u^2 R^2 G dg+2\int_M u^2 |\Ric|^2 G +2\int_M |\n u|^2 RG dg\\
&& +2 \int_M u^2 R\Delta Gdg  +4\int_M Ru\langle \n u, \n G \rangle dg \\
&=& 4 \int_M (\Ric-\n^2 f)(\n u, \n u) G dg +\int_M \left(2\Delta_f u -Ru \right)^2 G dg\\ 
&& +2\int_M u^2 |\Ric|^2 G dg +2\int_M u^2 R\Delta G dg+4\int_M Ru\langle \n u, \n G \rangle dg. 
\end{eqnarray*}
\end{proof}

\subsection{The nonnegative sectional curvature case}

We prove Theorem \ref{thm PF sec geq 0} and Corollary \ref{Corollary unique continuation} in this subsection.

\begin{proof}[Proof of Theorem \ref{thm PF sec geq 0}]
We need to estimate $I''(t)$ in \eqref{eq I''(t)} from below. 
Using Theorem \ref{thm matrix Harnack heat equation} and $\Ric \geq 0$, we get 
\begin{equation}\label{eq 1 sec}
\int_M (\Ric+\n^2 \log G)(\n u, \n u) G dg \geq -\frac{\k}{1-e^{-2\k t}} D(t). 
\end{equation}
By Corollary \ref{cor trace}, we have 
\begin{equation*}
\Delta G \geq |\n G|^2G^{-1} -\frac{n\k}{1-e^{-2\k t}} G. 
\end{equation*}
Since $R\geq 0$, we obtain
\begin{equation}\label{eq 2 sec}
\int_M u^2 R \Delta G \geq \int_M u^2 R |\n G|^2G^{-1}dg -\frac{n\k}{1-e^{-2\k t}}S(t).
\end{equation}
Noticing $0 \leq R\leq n\k$, we estimate that
\begin{eqnarray}\label{eq 3 sec}
&& 2 \int_M Ru\langle \n u, \n G \rangle dg \\ \nonumber
&\geq& -\int_M Ru^2 |\n G|^2G^{-1} dg -\int_M R |\n u|^2 G dg \\ \nonumber
&\geq& -\int_M Ru^2 |\n G|^2G^{-1} dg -n\k D(t)
\end{eqnarray}
Inserting the estimates \eqref{eq 1 sec}, \eqref{eq 2 sec}, and \eqref{eq 3 sec} into \eqref{eq I''(t)} yields
\begin{eqnarray}\label{eq I''(t) sec geq 0}
I''(t) \geq  \int_M \left(2\Delta_f u -Ru \right)^2 G dg-\left(\frac{2 \k}{1-e^{-2\k t}} +n\k \right)I'(t)
\end{eqnarray}
where we have used $S(t)\geq 0$ and \eqref{eq I derivative}. 

Using \eqref{eq I''(t) sec geq 0} and the Cauchy-Schwarz inequality, 
\begin{eqnarray*}
(I'(t))^2 &=& \left(\int_M (2\Delta_f u-Ru)G dg \right)^2 \\
&\leq&  \int_M u^2G dg \cdot  \int_M \left(2\Delta_f u -Ru \right)^2 G dg,
\end{eqnarray*}
we obtain that for $F(t):=(\log I(t))'$,
\begin{eqnarray*}
F'(t) &=& I(t)^{-2} \left( I''(t)-(I'(t))^2 \right) \\
&\geq & -\left(\frac{2 \k}{1-e^{-2\k t}} +n\k \right) F(t).
\end{eqnarray*}
Therefore, the quantity 
\begin{equation*}
e^{(n-2)\k t}(1-e^{-2\k t})F(t)
\end{equation*}
is monotone nondecreasing. 
\end{proof}

Next, we prove the unique continuation property. 
\begin{proof}[Proof of Corollary \ref{Corollary unique continuation}]

The key point to achieve unique continuation is that the correction factor in \eqref{eq PF F monotone sec geq 0} is asymptotic to $t$ as $t\to 0$. 

Suppose a solution $u=u(x,t)$ of the conjugate heat equation in $M \times [0, T)$ vanishes at infinity order at $(x_0, t_0) \in M\times (0,T)$. Since the quantity $e^{(n-2)\k t}(1-e^{-2\k t})F(t)$ is monotone nondecreasing on $[0,T]$, we have
\begin{equation*}
   (\log I(t))'=F(t) \geq F(T) e^{(n-2)\k (T-t)}\frac{1-e^{-2\k T}}{1-e^{-2\k t}} \geq \frac{C}{t}
\end{equation*}
for all $t\in (t_0,T)$, where $C=F(T)(1-e^{-2\k T})/(2\k)$. Hence, 
\begin{equation}\label{eq I(t) UC}
I(t) \geq  \left(\frac{t}{t_1} \right)^{C} I(t_1).
\end{equation}
The rest of the proof is standard since the heat kernel $G$ has Gaussian upper bound and the distance $d(x, x_0, t)$ are comparable in short time due to our assumption. This Gaussian bound and the assumption on the infinite vanishing order of $u$ at $(x_0, 0)$ implies, for all small $t>0$, 
\[
I(t) = \int_M u^2(x, t) G(x, x_0, t)dg(t)(x) \le C_N t^N
\]for any positive integer $N$, which is a contradiction to \eqref{eq I(t) UC} unless $u \equiv 0$.

\end{proof}

\subsection{The general case}

Since our assumption implies that $| Ric | \le c_n K$, taking $\alpha=2$ and $\rho=\infty$ in 
\cite[Theorem 2.7]{BCP10}, we have the following Li-Yau bound
\begin{equation}
    \frac{|\n G|^2}{G^2} - 2 \frac{G_t}{G} \leq c(n) \left(\frac{1}{t} + K \right),
\end{equation} 
where $c(n)$ is a dimensional constant. Then 
\begin{equation}
\label{ddg=gt>}
    \Delta G =G_t \geq \frac{|\n G|^2}{2 G} - \left( \frac{c(n)}{t} +K \right) G.
\end{equation} Note that this also follows from the matrix Harnack inequality in Section 4 after taking the trace.

Let $s_0 = \inf_{x \in M^n} R(x, 0)$. It is well known that $R(x, t) \ge s_0$ for all $t\in [0,T]$.
Since $R-s_0 \ge 0$ and 
\begin{eqnarray}\label{eq 5.17}
&&\int_M  u^2 R\Delta G dg + 2\int_M R u\langle \n u, \n G \rangle dg\\ \nonumber 
&=&\int_M  u^2 (R-s_0) \Delta G dg + \int_M  u^2 s_0 \Delta G dg+ 2\int_M R u\langle \n u, \n G \rangle dg\\ \nonumber
&=&\int_M  u^2 (R-s_0) \Delta G dg +  2\int_M (R - s_0) u\langle \n u, \n G \rangle dg,
\end{eqnarray} 
we can apply \eqref{ddg=gt>} on the right-hand hand side of \eqref{eq 5.10} to reach 
\begin{eqnarray*}
&& \int_M u^2 R\Delta G dg + 2\int_M R u\langle \n u, \n G \rangle dg\\
&\ge&  \int_M  u^2 (R-s_0) \left[ \frac{|\n G|^2}{2 G} - \left( \frac{c(n)}{t} +K \right) G \right] dg \\
&& + 2\int_M (R - s_0) u\langle \n u, \n G \rangle dg.
\end{eqnarray*}
Using  Cauchy-Schwarz inequality, we have 
\begin{eqnarray*}
&& 2 \int_M (R-s_0) u\langle \n u, \n G \rangle dg \\
&\geq&  - \frac{1}{2} \int (R-s_0) u^2 |\n \log G|^2 G dg 
 - 2 \int_M (R-s_0) |\n u|^2 G dg.
\end{eqnarray*} 
Hence, we obtain the estimate 
\begin{eqnarray*}
&& \int_M R  u^2 \Delta G dg 2 + \int_M R u\langle \n u, \n G \rangle dg \\
& \geq&  -\left( \frac{c(n)}{t} +K \right)  \int_M   (R-s_0) u^2 G  dg \\
&& - 2 \int_M (R-s_0) |\n u|^2 G dg.
\end{eqnarray*}

Using this inequality, the matrix Harnack inequality in Theorem \ref{thm matrix Harnack heat equation general case},
\begin{equation}\label{defc1t}
- \n^2 f \geq - c_1(t) g_{ij}, \quad c_1(t) := \frac{1}{2 t} + \frac{1}{t} \beta(t, n, K)
\end{equation} 
we deduce, from Lemma \ref{lemI''t}, that
\begin{eqnarray*}
I''(t) &\geq & \int_M \left( 2\Delta_fu -Ru\right)^2 G dg+2\int u^2 |\Ric|^2 G dg\\
&& -4(c_1(t) + c(n) K) \int_M |\n u|^2 G dg -4 \int_M (R-s_0) |\n u|^2 G dg \\
&& - 2 \left( \frac{c(n)}{t} +K \right) \int_M (R-s_0) u^2 G dg.
\end{eqnarray*} 
Here we have used $\Ric(\n u, \n u) \ge -c(n) K |\n u |^2$. 
Using the fact that $R-s_0 \le c(n) K$ and adjusting the dimensional constant $c(n)$, we deduce
\begin{eqnarray*}
I''(t) &\geq & \int_M \left( 2\Delta_fu -Ru\right)^2 G dg+2\int u^2 |\Ric|^2 G dg\\
&& -2(c_1(t) + c(n) K) \left[\int_M  2|\n u|^2 G dg +  \int_M  R u^2 G dg \right]\\
&& + (c_1(t) + c(n) K) \int_M  R u^2 G dg \\
&& - 2 \left( \frac{c(n)}{t} +K \right) \int_M (R-s_0) u^2 G dg.
\end{eqnarray*} 
Therefore
\begin{eqnarray*}
I''(t) &\geq & \int_M \left( 2\Delta_fu -Ru\right)^2 G dg+2\int u^2 |\Ric|^2 G dg\\
&& -2(c_1(t) + c(n) K) \left[\int_M  2|\n u|^2 G dg +  \int_M  R u^2 G dg \right]\\
&& - c_2(t) \int_M  u^2 G dg, 
\end{eqnarray*}
where 
\begin{equation}\label{defc2t}
c_2(t)= c(n) K \left[2 (c_1(t) + c(n) K) + 2 \left( \frac{c(n)}{t} +K \right) \right].
\end{equation} This implies that
\begin{eqnarray*}
&& I^2(t)(\log I(t))'' \\
&=& I''(t) I(t) -(I'(t))^2 \\
&\geq& - 2  (c_1(t) + c(n) K) I(t) \left(2 \int_M |\n u|^2 G dg  + \int_M  u^2G dg \right)\\
&& -  c_2(t) I^2(t)+ I(t) \int_M \left( 2\Delta_fu -Ru\right)^2 G dg-(2 D(t) + S(t))^2.
\end{eqnarray*} 
Using integration by parts, we see that
\[
2D(t) + S(t)=  \int_M (-2 \Delta_f u + R u) u G dg.
\]
Therefore the difference of the last two terms in the preceding inequality is non-negative by Cauchy-Schwarz inequality, giving us:
\begin{equation}
I^2(t)(\log I(t))'' \geq - 2 (c_1(t) + c(n) K)  I(t) \left[ 2 D(t) + S(t)  \right] - c_2(t) I^2(t).
\end{equation} Hence 
\begin{equation}
\label{log''>-c/t}
(\log I(t))'' \geq - 2 (c_1(t) + c(n) K)  \frac{ 2 D(t) + S(t)}{I(t)} - c_2(t).
\end{equation} Let us recall from \eqref{defc1t} and Theorem
\ref{thm matrix Harnack heat equation general case},
\begin{eqnarray}\label{defc1t2}
c_1(t) &:=& \frac{1}{2 t} + \frac{1}{t} \beta(t, n, K) \\ \nonumber
&=&  \frac{1}{2 t} + \frac{1}{t} \left[4 \sqrt{n K t}  +C_2 (K+1)t +C_1 \sqrt{K t}\,  \diam \right]
\end{eqnarray}
and from \eqref{defc2t}
\[
c_2(t)= c(n) K \left[2 (c_1(t) + c(n) K) + 2 \left( \frac{c(n)}{t} +K \right) \right].
\]

Let 
\[
Z_0 = \sup_{t \in (0, T]} [t c_2(t)].
\]
Inequality \eqref{log''>-c/t} infers for $F(t):=(\log I(t))'= \frac{ 2 D(t) + S(t)}{I(t)}$, 
\begin{equation}
\label{F't>-Ft}
F'(t) \ge -2(c_1(t) + c(n) K) F(t) - \frac{Z_0}{t}.
\end{equation}
Let 
\begin{equation}
\label{defppnk}
p = p(n, K, v_0, T, diam)= \sup_{t \in (0, T]} [t 2 (c_1(t) + c(n) K)]
\end{equation} From the preceding inequality, we see that 
\[
t F'(t) \ge -p  F(t) - Z_0.
\]Therefore 
\[
\left( t^p F(t) + \frac{Z_0}{p} t^p \right)' \ge 0
\]
Thus, we have proved Theorem \ref{thm PF general}. 

\subsection{The unweighted case}
For a solution $u(x,t)$ to the heat equation on $\R^n$, the monotonicity of the unweighted frequency 
\begin{equation*}
    \frac{\int_{\R^n} |\n u(x,t)|^2 dx}{\int_{\R^n} u^2(x,t) dx}
\end{equation*}
is equivalent to the log convexity of the energy $\int u^2 dx$. This is a classical result that can be used to prove the uniqueness of the backward heat equation (see for instance \cite{JohnPDEbook1982}). Here we extend this result to the conjugate heat equation coupled with the Ricci flow. 
Compared with the weighted case in this section, the curvature assumption is $\Ric \ge 0$ and no upper bound of any curvature is needed. 
The unweighted monotonicity, however, is not strong enough to prove the unique continuation property.

\begin{theorem}\label{thm PF no weight}
Let $(M^n,g(t))$, $t\in [0,T]$, be a compact Ricci flow. Let $u$ be a solution to the backward conjugate heat equation \eqref{backward conjugate heat equation}. Define 
$$I(t) = \int u^2(x,t) d\mu_{g(t)}.$$
If $(M^n,g(t))$ has nonnegative Ricci curvature, then 
$$(\log I(t))'' \ge 0.$$
\end{theorem}

\begin{proof}
As before, we also define
\begin{eqnarray*}
D(t) &=& \int |\nabla u(x,t)|^2 d\mu_{g(t)}, \\
S(t) &=& \int u^2(x,t) R(x,t) d\mu_{g(t)},
\end{eqnarray*}
and write $I(t)=\int_M u^2 dg$ for short and similar notations for other integrals. 

By direct computation as for the weighted case, we have
\begin{eqnarray*}
I'(t) &=&  2 \int |\nabla u|^2dg +\int u^2 R dg=2D(t)+S(t),
\end{eqnarray*}
\begin{eqnarray*}
D'(t) &=& 2\int \langle \nabla u, \nabla u_t \rangle dg +2\int \Ric(\nabla u, \nabla u)dg - \int |\nabla u|^2 R dg \\
&=& -2\int u_t \Delta u dg+2\int \Ric(\nabla u, \nabla u)dg  - \int |\nabla u|^2 R dg,
\end{eqnarray*}
and 
\begin{eqnarray*}
S'(t) &=&  2 \int uu_t R dg+\int u^2 (\Delta R +2|\Ric|^2) dg-\int u^2 R^2 dg \\
&=&  2 \int uu_t R dg+\int R(2u\Delta u  +2|\nabla u|^2) \\
&& +2\int u^2|\Ric|^2 dg-\int u^2 R^2 dg.\\
\end{eqnarray*}
Using $\Delta u=Ru-u_t$, we deduce
\begin{eqnarray*}
I''(t) &=& 2D'(t)+S'(t) \\
&=&  -4\int u_t \Delta u dg+4\int \Ric(\nabla u, \nabla u) dg \\
&& +2\int Ru \Delta u dg +2 \int uu_t R +2\int u^2|\Ric|^2 dg -\int u^2 R^2 dg\\
&=& 4\int u_t^2 dg-4\int uu_t R dg+\int u^2R^2 dg \\
&& +4\int \Ric(\nabla u, \nabla u) dg+2\int u^2|\Ric|^2 dg \\
&=& \int (2u_t-uR)^2 dg +4\int \Ric(\nabla u, \nabla u) dg+2\int u^2|\Ric|^2 dg.
\end{eqnarray*}
Using $I'(t)=\int u(2u_t-uR)$, we then get
\begin{eqnarray*}\label{eq 5.25}
&& I^2 (\log I)'' \\ \nonumber
&=&  \int u^2 dg\cdot \int (2u_t -uR)^2 dg- \left(\int u(2u_t-uR) dg\right)^2 \\ \nonumber
&& +\int u^2 dg\left( 4\int \Ric(\nabla u, \nabla u) dg 
 +2\int u^2|\Ric|^2  dg\right)
\end{eqnarray*}
The first line on the right-hand side of the above equation is nonnegative by the Cauchy-Schwarz inequality and the second line is nonnegative since $\Ric \geq 0$. Therefore, we have proved the log convexity of the energy $I(t)$. 
\end{proof}
\section{Matrix Harnack for the Conjugate Heat Equation}
 

In this section, we prove Theorem \ref{thm matrix Harnack backward conjugate heat equation}.
Let's first recall the Harnack estimate for the Ricci flow since it will be used in the proof. 
\begin{proposition}
Let $(M^n,g(t))$, $t\in (0,T)$, be a complete solution to the Ricci flow with bounded nonnegative complex sectional curvature. Define 
\begin{equation}\label{eq def M_ij}
    M_{ij} := \Delta R_{ij} -\tfrac{1}{2}\n_i\n_j R +2R_{ikjl}R_{kl} -R_{ik}R_{jk} +\tfrac{1}{2t}R_{ij},
\end{equation}
and 
\begin{equation}\label{eq def P_kij}
    P_{kij} := \n_k R_{ij} -\n_i R_{kj}. 
\end{equation}
Then we have
\begin{equation}\label{eq Harnack}
    M(w,w)+2P(v,w,w)+\Rm(v,w,v,w) \geq 0
\end{equation}
for all $(x,t) \in M \times (0,T)$ and all vectors $v,w \in T_xM$. 
\end{proposition}

\begin{proof}
The Harnack estimate for the Ricci flow was originally proved by Hamilton \cite{Hamilton93JDG} under the nonnegative curvature operator condition. This version stated here is a generalization due to Brendle \cite{Brendle09}. Notice that $M$ has nonnegative complex sectional curvature if and only $M\times \R^2$ has nonnegative isotropic curvature, which is an observation of Ni and Wolfson \cite{NWolfson07}.
\end{proof}

The next step is to derive an evolution inequality for 
\begin{equation}\label{eq def Z_ij}
    Z_{ij}:=R_{ij} -\n_i\n_j \log u -\eta(t) g_{ij}.
\end{equation}
\begin{proposition}\label{prop evolution Z_ij}
Let $(M^n,g(t))$, $u$, and $\eta(t)$ be the same as in Theorem \ref{thm matrix Harnack backward conjugate heat equation}. Then $Z_{ij}$ defined in \eqref{eq def Z_ij} satisfies 
\begin{equation}\label{eq Z_ij estimate}
\tfrac{1}{2}(\p_t+\Delta)Z_{ij} 
\geq  Z^2_{ij}-R_{ikjl}Z_{kl}-\tfrac{1}{2}R_{ik}Z_{jk}-\tfrac{1}{2}R_{jk}Z_{ik}+2\eta Z_{ij}. 
\end{equation}
\end{proposition}

\begin{proof}[Proof of Proposition \ref{prop evolution Z_ij}]
For simplicity, we write $v=\log u$ and
\begin{equation*}
    H_{ij} = \n_i \n_j \log u.
\end{equation*}
Then the calculations in Section 2 with $\ve=-1$ and $\delta =1$ apply to this setting and we obtain from \eqref{eq evolution H_ij} that
\begin{eqnarray}\label{eq evolution H_ij conjugate}
 (\p_t+\Delta)H_{ij} &=&  \n_i\n_j R - (2R_{ikjl}H_{kl}-R_{ik}H_{jk}-R_{jk}H_{ik}) \\ \nonumber
&& - 2 \left(H^2_{ij} +R_{ikjl}\n_k v \n_l v +\n_k H_{ij} \n_k v \right)
\\ \nonumber
&& +2 (\n_iR_{jk}+\n_jR_{ik}-\n_kR_{ij})\n_k v.  
\end{eqnarray}
Under the Ricci flow, we have (see \cite[page 112]{CLN})
\begin{equation*}
    (\p_t - \Delta) R_{ij} =2R_{ikjl}R_{kl}-2R_{ik}R_{jk}.
\end{equation*}
Hence, 
\begin{equation}\label{eq evolution R_ij conjugate}
    (\p_t + \Delta) R_{ij} =2\Delta R_{ij} + 2R_{ikjl}R_{kl}-2R_{ik}R_{jk}. 
\end{equation}
We also notice that
\begin{equation}\label{eq evolution cg_ij conjugate}
    (\p_t + \Delta) (\eta(t)g_{ij}) = \eta'(t)g_{ij} -2\eta(t)R_{ij}.
\end{equation}
Combining \eqref{eq evolution H_ij conjugate}, \eqref{eq evolution R_ij conjugate}, and \eqref{eq evolution cg_ij conjugate} together, we obtain that 
\begin{eqnarray*}
&& \tfrac{1}{2}(\p_t+\Delta)Z_{ij} \\
&=& \tfrac{1}{2} (\p_t+\Delta)R_{ij}-\tfrac{1}{2}(\p_t+\Delta )H_{ij}-\tfrac{1}{2}(\p_t+\Delta )(\eta(t)g_{ij}) \\
&=& \Delta R_{ij} +R_{ikjl}R_{kl}-R_{ik}R_{jk} -\tfrac{1}{2}\n_i\n_j R + H^2_{ij} \\
&&   +R_{ikjl}H_{kl}-\tfrac{1}{2}R_{ik}H_{jk}-\tfrac{1}{2}R_{jk}H_{ik} +R_{ikjl}\n_k v \n_l v \\
&& -\n_k (Z_{ij}-R_{ij})\n_k v - (\n_iR_{jk}+\n_jR_{ik}-\n_kR_{ij})\n_k v \\
&& -\tfrac{\eta'(t)}{2}g_{ij} +\eta(t)R_{ij} \\
&=& \Delta R_{ij} -\tfrac{1}{2} \n_i \n_j R +2R_{ikjl}R_{kl}-R_{ik}R_{jk} +R_{ikjl}\n_k v \n_l v \\
&& + H^2_{ij} -R_{ikjl}R_{kl} + R_{ikjl}H_{kl}-\tfrac{1}{2}R_{ik}H_{jk}-\tfrac{1}{2}R_{jk}H_{ik}\\
&&  -\n_k Z_{ij}\n_k v  - (\n_iR_{jk}+\n_jR_{ik}-2\n_kR_{ij})\n_k v \\
&& -\tfrac{\eta'(t)}{2}g_{ij} +\eta(t)R_{ij}.
\end{eqnarray*}
Using \eqref{eq def M_ij}, \eqref{eq def P_kij}, and
\begin{eqnarray*}
&& -(\n_iR_{jk}+\n_jR_{ik}-2\n_kR_{ij})\n_k v \\
&=& (\n_k R_{ij} -\n_i R_{jk})\n_k v +(\n_k R_{ij} -\n_j R_{ik})\n_k v \\
&=& (P_{kij}+P_{kji})\n_k v,
\end{eqnarray*}
we get 
\begin{eqnarray}\label{eq 5.10}
&& \tfrac{1}{2}(\p_t+\Delta)Z_{ij} \\ \nonumber
&=& M_{ij}-\tfrac{1}{2t}R_{ij}+(P_{kij}+P_{kji})\n_k v+ R_{ikjl}\n_k v \n_l v -\n_k Z_{ij}\n_k v \\ \nonumber
&& +H^2_{ij}-R_{ikjl}R_{kl} + R_{ikjl}H_{kl} +\tfrac{1}{2}R_{ik}H_{jk}+\tfrac{1}{2}R_{jk}H_{ik} \\ \nonumber && -\tfrac{\eta'(t)}{2}g_{ij} +\eta(t)R_{ij}.
\end{eqnarray}
Next, we compute using \eqref{eq def Z_ij} that
\begin{eqnarray}\label{eq 5.11}
&& H^2_{ij}-R_{ikjl}R_{kl} + R_{ikjl}H_{kl} +\tfrac{1}{2}R_{ik}H_{jk}+\tfrac{1}{2}R_{jk}H_{ik} \\ \nonumber
&=& (R_{ik}-Z_{ik}-\eta g_{ik})(R_{jk}-Z_{jk}- \eta g_{jk}) -R_{ikjl}(Z_{kl}+ \eta g_{kl}) \\ \nonumber
&& -\tfrac{1}{2}R_{ik}(R_{jk}-Z_{jk}-\eta g_{jk})-\tfrac{1}{2}R_{jk}(R_{ik}-Z_{ik}-\eta g_{jk}) \\ \nonumber
&=& Z^2_{ij}-R_{ikjl}Z_{kl}-\tfrac{1}{2}R_{ik}Z_{jk}-\tfrac{1}{2}R_{jk}Z_{ik}+2\eta Z_{ij}+\eta ^2g_{ij}-2\eta R_{ij}.
\end{eqnarray}
Substituting \eqref{eq 5.11} into \eqref{eq 5.10} produces
\begin{eqnarray}\label{eq evolution Z_ij conjugate}
&& \tfrac{1}{2}(\p_t+\Delta)Z_{ij} \\ \nonumber 
&=& M_{ij}+(P_{kij}+P_{kji})\n_k v+ R_{ikjl}\n_k v \n_l v -\n_k Z_{ij}\n_k v \\ \nonumber
&& +Z^2_{ij}-R_{ikjl}Z_{kl}-\tfrac{1}{2}R_{ik}Z_{jk}-\tfrac{1}{2}R_{jk}Z_{ik}+2cZ_{ij}\\  \nonumber
&& +\eta^2g_{ij}-\eta R_{ij}-\tfrac{\eta'(t)}{2}g_{ij}-\tfrac{1}{2t}R_{ij},
\end{eqnarray}
Using \eqref{eq c(t) ODE} and $\Ric \leq \kappa g$, we have 
\begin{eqnarray}\label{eq estimate c}
 && \eta^2g_{ij}-\eta R_{ij}-\tfrac{\eta'(t)}{2}g_{ij}-\tfrac{1}{2t}R_{ij} \\ \nonumber
&\geq&   \tfrac{1}{2t}(\kappa g_{ij}-R_{ij}) +\eta(\kappa g_{ij}-R_{ij}) .
\end{eqnarray}
In view of \eqref{eq Harnack} and \eqref{eq estimate c}, we conclude that 
\begin{equation*}
\tfrac{1}{2}(\p_t+\Delta)Z_{ij} 
\geq  Z^2_{ij}-R_{ikjl}Z_{kl}-\tfrac{1}{2}R_{ik}Z_{jk}-\tfrac{1}{2}R_{jk}Z_{ik}+2\eta Z_{ij}. 
\end{equation*}
Hence, \eqref{eq Z_ij estimate} is proved.
\end{proof}

We are ready to prove Theorem \ref{thm matrix Harnack backward conjugate heat equation}.
\begin{proof}[Proof of Theorem \ref{thm matrix Harnack backward conjugate heat equation}]
By Proposition \ref{prop evolution Z_ij}, 
\begin{equation}\label{eq 5.12}
\tfrac{1}{2}(\p_t+\Delta)Z_{ij} 
\geq  Z^2_{ij}-R_{ikjl}Z_{kl}-\tfrac{1}{2}R_{ik}Z_{jk}-\tfrac{1}{2}R_{jk}Z_{ik}+2\eta Z_{ij}. 
\end{equation}

If $(M, g(t))$ is compact, it follow from $\eta(t)\to \infty$ as $t\to T$ that $Z_{ij} \leq 0$ as $t \to T^-$. 
Noticing that each term on the right-hand side of \eqref{eq 5.12}
satisfies the null-eigenvector condition, we conclude using Hamilton's tensor maximum principle (see \cite{Hamilton86} or \cite[Theorem 3.3]{CLN}) that $Z_{ij} \leq 0$ on $M\times (0,T)$. 

When $(M, g(t))$ is complete noncompact, one can proceed as in subsection 3.2 and use the maximum principle (see \cite[Theorem 12.22]{Chowbookpart2} to prove the estimate. We omit the technical details here.
\end{proof}

Next, we prove Corollary \ref{corollary matrix Harnack conjugate}. 
\begin{proof}[Proof of Corollary \ref{corollary matrix Harnack conjugate}]
Note that the function
\begin{equation*}
   \eta_0(t)=\frac{\k}{1-e^{-2\k(T-t)}} +\sqrt{\frac{\k}{2t}}
\end{equation*}
satisfies $\eta_0' \leq 2\eta_0^2 -2\k \eta_0 -\frac{\k}{t}$ and $\eta_0(t)\to \infty$ as $t\to T$. 
Hence, the inequality \eqref{eq matrix estimate KC geq 0} follows from choosing $\eta=\eta_0$ in \eqref{eq matrix estimate KC geq 0 eta}.
\end{proof}

Below we present an improvement of Theorem \ref{thm matrix Harnack backward conjugate heat equation} when the Ricci flow is ancient. 
\begin{theorem}\label{thm matrix LYH ancient}
Let $(M^n,g(t))$, $t\in (-\infty,T)$, be a compact ancient solution to the Ricci flow with bounded nonnegative complex sectional curvature. Let $u: M \times [t_0, T] \to (0,\infty)$, $-\infty \leq t_0<T$, be a positive solution to the backward conjugate heat equation $u_t+\Delta_{g(t)} u=Ru$. Suppose that $\Ric(x,t) \leq \k(x,t) g$ for some $\k>0$ and for all $(x,t)\in M\times [t_0,T]$. 
Then we have
\begin{equation*}
\Ric -\n^2 \log u -\frac{\kappa}{1-e^{-2\kappa(T-t)}} g \leq 0,
\end{equation*}
for all $(x,t) \in M \times (t_0,T)$. 
\end{theorem}

\begin{remark}
In Theorem \ref{thm matrix LYH ancient}, it suffices to assume the weaker condition that $M\times \R$ has nonnegative isotropic curvature, in view of \cite{BCW19} and \cite[Proposition 6.2]{LN20}. 
\end{remark}

\begin{proof}[Proof of Theorem \ref{thm matrix LYH ancient}]
The proof is almost identical to that of Theorem \ref{thm matrix Harnack backward conjugate heat equation}. The difference is that we get the improved Harnack estimate (compared with \eqref{eq Harnack}) 
\begin{equation*}
M(w,w)+2P(v,w,w)+\Rm(v,w,v,w) \geq \tfrac{1}{2t}\Ric(w,w)
\end{equation*}
on ancient Ricci flows (see \cite{Hamilton93JDG} or \cite{Brendle09}). As a result, the ordinary differential inequality for $\eta(t)$ becomes 
\begin{equation}\label{ODE eta ancient}
    \eta'\leq 2\eta^2 -2\kappa \eta. 
\end{equation}
The theorem follows from the fact that the function
\begin{equation*}
    \eta(t):=\frac{\kappa}{1-e^{-2\kappa(T-t)}}
\end{equation*}
solves \eqref{ODE eta ancient} with equality on $(t_0,T)$ and satisfies $\eta(t) \to \infty$ as $t\to T$. 
\end{proof}

We get space-time gradient estimates for $\log u$ by tracing the matrix Li-Yau-Hamilton estimates. 
\begin{corollary}
Let $(M^n,g(t))$ and $u$ be the same as in Theorem \ref{thm matrix Harnack backward conjugate heat equation}. Then
\begin{equation*}
    R-\Delta \log u - \frac{n\k}{1-e^{-2\k(T-t)}} -n\sqrt{\frac{\k}{2t}} \leq 0
\end{equation*}
for all $(x,t)\in M\times (0,T)$. 
\end{corollary}

\begin{corollary}
Let $(M^n,g(t))$ and $u$ be the same as in Theorem \ref{thm matrix LYH ancient}. Then
\begin{equation*}
    R-\Delta \log u - \frac{n\k}{1-e^{-2\k(T-t)}} \leq 0
\end{equation*}
for all $(x,t)\in M\times (0,T)$.  
\end{corollary}
Classical-type Harnack inequalities follow from integrating the above estimates. It is an interesting question whether the above gradient estimates hold under nonnegative Ricci or sectional curvature.

As an application of Theorem \ref{thm matrix Harnack backward conjugate heat equation}, we prove Proposition \ref{prop PF conjugate}. 
\begin{proof}[Proof of Proposition \ref{prop PF conjugate}]
By \eqref{eq matrix estimate KC geq 0}, we have $\Ric-\n^2 \log u \leq \frac{k(t)}{2(T-t)}$ with 
\begin{equation*}
\frac{k(t)}{(T-t)}= \frac{2\k}{1-e^{-2\k (T-t)}} +\sqrt{\frac{2\k}{t}}. 
\end{equation*}
By the work of Baldauf and Kim \cite{BK22}, the correction factor is given by
\begin{equation*}
e^{\int \frac{1-k(t)}{T-t} dt} = \frac{1}{T-t} e^{\int \left(\frac{-2\k}{1-e^{-2\k (T-t)}} -\sqrt{2\k} \frac{1}{\sqrt{t}} \right) dt } =\frac{1}{T-t} e^{2\k (T-t)-1} e^{-\sqrt{8\k t}}.
\end{equation*}
Therefore, we have proved Proposition \ref{prop PF conjugate}. 
\end{proof}

Proposition \ref{prop PF conjugate} implies a unique continuation result.
\begin{corollary}\label{corollary UC KC geq 0}
Let $(M^n,g(t))$, $t\in [0,T]$, be a solution to the Ricci flow with nonnegative complex sectional curvature and $\Ric \leq \k g$ for some $\k>0$. Suppose that a solution $u(x,t)$ of the heat equation \eqref{heat equation} on $M\times [0,T]$ vanishes of infinity order at $(x_0, t_0) \in M\times (0,T)$. Then $u\equiv 0$ in $M\times [0,T]$. 
\end{corollary}

\begin{proof}[Proof of Corollary \ref{corollary UC KC geq 0}]
The key point is that the correction factor in \eqref{eq PF KC geq 0} is asymptotic to $(T-t)$ as $t \to T$. The proof is similar to that of Corollary \ref{Corollary unique continuation} and we omit the details. 
\end{proof}

\section{An improvement of Hamilton's matrix estimate}

We present the proof of Theorem \ref{thm improve Hamilton} in this section. 

\begin{proof}[Proof of Theorem \ref{thm improve Hamilton}]
We write $v=\log u$ and $H_{ij}=\n_i\n_j \log u$. By a straightforward calculation as in Section 2 or \cite{Hamilton93}, we derive that
\begin{eqnarray*}
 (\p_t -\Delta)H_{ij} 
&=& 2H^2_{ij}+2R_{ikjl}H_{kl}-R_{ik}H_{jk}-R_{jk}H_{ik}+2R_{ikjl}\n_k v \n_l v \\
&&  +2\n_k H_{ij}\n_k v +(\n_l R_{ij} -\n_iR_{jl}-\n_j R_{il})\n_l v.
\end{eqnarray*}
Compared to \eqref{eq evolution H_ij S3} in the Ricci flow case, here we have the addition term $(\n_l R_{ij} -\n_iR_{jl}-\n_j R_{il})\n_l v$. 
As in the proof of Theorem \ref{thm matrix Harnack heat equation general case} in Section 4, we consider $Q_{ij}:=tH_{ij}$, which satisfies
\begin{eqnarray}\label{eq 7.0}
 (\p_t -\Delta)Q_{ij} 
&=& \frac{1}{t}Q_{ij}+\frac{2}{t}Q^2_{ij}+2R_{ikjl}Q_{kl}-R_{ik}Q_{jk}-R_{jk}Q_{ik}\\ \nonumber 
&& +2tR_{ikjl}\n_k v \n_l v  +2\n_k Q_{ij}\n_k v \\ \nonumber 
&&+t(\n_l R_{ij} -\n_iR_{jl}-\n_j R_{il})\n_l v.
\end{eqnarray}
Let $\l_1$ be the smallest eigenvalue of $Q_{ij}$. 
Using $|\n \Ric|\leq L$, we estimate the last term 
\begin{equation*}
t(\n_l R_{ij} -\n_iR_{jl}-\n_j R_{il})\n_l v \geq -3Lt |\n v| g_{ij} \geq -3(L^{\frac{2}{3}}|\n v|^2 +L^{\frac{4}{3}})g_{ij}.
\end{equation*}
As in the proof of Theorem \ref{thm matrix Harnack heat equation general case}, we have the estimates
\begin{eqnarray*}
R_{ikjl}Q_{kl}\xi_i\xi_j &\geq& (2n-1)K\l_1 -Kt\Delta v, \\
R_{ikjl}\n_k v \n_l v \xi_i\xi_j &\geq & -K |\n v|^2, 
\end{eqnarray*}
Therefore, we deduce from \eqref{eq 7.0} that at a negative minimum point $(x_0,t_0)\in M\times [0,t_0]$ of $\l_1$, 
\begin{eqnarray}\label{eq 7.1}
2\l_1^2 &\leq&  -\l_1-2(2n-1)Kt\l_1+2Kt^2\Delta v +2Kt^2|\n v|^2 \\ \nonumber
&& +3L^{\frac{2}{3}}t^2|\n v|^2 +3L^{\frac{4}{3}}t^2. 
\end{eqnarray}

From now on, we assume $u=G(x,t,y)$ is the heat kernel and estimate $\Delta \log u$ and $|\n \log u|$. 
First, applying Hamilton's gradient bound \cite[Theorem 1.1]{Hamilton93} to $u(x,t+t_0,y)$ yields for any $t_0>0$ that
\begin{equation}\label{eq 7.2}
(t-(t_0/2))|\n \log u|^2 \leq (1+2(n-1)K(t-t_0/2))\log \frac{A}{u},
\end{equation}
where $A=\sup \{u(x,t):(x,t)\in M\times [t_0/2,t_0]\}$. 
According to \cite{Zhang21}, we have 
\begin{eqnarray}\label{eq 7.3}
\log \frac{A}{u} &\leq& 2 \log C_1 +4C_2Kt_0+\frac{d^2(x,y)}{3t_0}+C_3\sqrt{K}\frac{\Diam}{\sqrt{t_0}}\\ \nonumber
&\leq & C_4(1+K+Kt_0)+\frac{\Diam^2}{2t_0},
\end{eqnarray}
where $C_1$, $C_2$, $C_3$, and $C_4$ are dimensional constants.
Substituting \eqref{eq 7.3} into \eqref{eq 7.2} produces
\begin{equation}\label{eq 7.35}
t_0|\n \log u|^2 \leq 2(1+(n-1)Kt_0)\left(C_4(1+K+Kt_0)+\frac{\Diam^2}{2t_0}\right).
\end{equation}
Next, we apply the Laplacian estimate (see \cite[Theorem E.35]{Chowbookpart2}) to the function $u(x,t+t_0,y)$ and get 
\begin{equation*}
(t-t_0/2)\left( \Delta \log u +2|\n \log u|^2 \right) \leq (1+(n-1)K(t-t_0/2))\left(n+4\log \frac{A}{u}\right). 
\end{equation*}
At $t=t_0$, we obtain using \eqref{eq 7.3} that
\begin{eqnarray}\label{eq 7.4}
&& t_0(\Delta \log u +2|\n \log u|^2) \\ \nonumber
&\leq&  (2+(n-1)Kt_0)\left(n+4C_4(1+K+Kt_0)+ 2\frac{\Diam^2}{t_0}\right)
\end{eqnarray}
Since $t_0>0$ is arbitrary, we conclude that \eqref{eq 7.35} and \eqref{eq 7.4} are valid for any $t_0=t>0$.
Combining \eqref{eq 7.35} and \eqref{eq 7.4}, we estimate that
\begin{eqnarray*}
B&:=& 2Kt \Delta v +2Kt |\n v|^2 +3L^{\frac{2}{3}}t|\n v|^2 \\ \nonumber
&\leq & 2K(2+(n-1)Kt_0)\left(n+4C_4(1+K+Kt_0)+ 2\frac{\Diam^2}{t}\right) \\ \nonumber
&& +3L^{\frac{2}{3}}(2(1+(n-1)Kt_0))\left(C_4(1+K+Kt_0)+\frac{\Diam^2}{2t}\right) \\ \nonumber
&\leq & 2nK(2+(n-1)Kt)  +C_5(K+L^{\frac{2}{3}})(1+Kt)(1+K+Kt)\\ \nonumber
&& + \left(4K(2+(n-1)Kt)+3L^{\frac{2}{3}}(1+(n-1)Kt) \right)\frac{\Diam^2}{t}, 
\end{eqnarray*}
where $C_5$ depends only on the dimension. Now, \eqref{eq 7.1} implies, 
\begin{equation}\label{eq 7.6}
2\l_1^2 +\l_1+2(2n-1)Kt\l_1 \leq tB +3L^{\frac{4}{3}}t^2. 
\end{equation}
Note that if $ax^2+bx \leq c$ with $a>0$, $b>0$, and $c>0$, then we have the lower bound
\begin{equation*}
    x \geq \frac{-b-\sqrt{b^2+4ac}}{2a} \geq -\frac{b+\sqrt{ac}}{a}.
\end{equation*}
Therefore, we deduce from \eqref{eq 7.6} that 
\begin{eqnarray*}
\l_1 &\geq& -\frac{1}{2}+(2n-1)Kt + \frac{1}{2}\sqrt{2(tB+3L^{\frac{4}{3}}t^2)}\\
&\geq & -\frac{1}{2}+(2n-1)Kt +\frac{1}{2}\sqrt{tB}+\frac{\sqrt{3}}{2}L^{\frac{2}{3}}t.
\end{eqnarray*}
The desired estimate for $u=G(x,t,y)$ then follows by noting 
\begin{eqnarray*}
\sqrt{tB}\geq t \gamma(t,n,K,L).
\end{eqnarray*}
Finally, we can follow the same argument in Section 4 or \cite{Zhang21} to show that the desired estimate holds from any positive solution to the heat equation. 
This completes the proof.


\end{proof}

\section*{Acknowledgments}
The authors would like to thank Dr. Jianyu Ou for his interest and for pointing out \cite{MO22}. The second author thanks the Department of Mathematics, Statistics and Physics at Wichita State University for its hospitality during his visit, at which occasion the initial questions for this project were discussed.

\section*{Conflict of Interest}
On behalf of all authors, the corresponding author states that there is no conflict of interest.

\section*{Data Availability Statement}
Data sharing is not applicable to this article as no datasets were generated or analyzed during the current study.

\bibliographystyle{alpha}
\bibliography{references}

\end{document}